\documentclass[preprint,none]{imsart}
\RequirePackage[OT1]{fontenc}
\usepackage{color,amsthm,amsmath,amssymb,natbib,enumerate}
\newcommand{\mc}{\mathcal}
\newcommand{\mb}{\mathbb}

\newcommand{\E}{\mathbb{E}}
\newcommand{\R}{\mathbb{R}}
\newcommand{\I}{\textbf{1}_}
\newcommand{\ti}{\widetilde}
\newcommand{\ol}{\overline}

\newcommand{\eps}{\varepsilon}
\newcommand{\lo}{\langle}
\newcommand{\ro}{\rangle}
\renewcommand{\a}{\alpha}
\newcommand{\an}{|\a|_1}
\DeclareMathOperator*\sgn{sgn}

\newcommand{\tr}{{\text{\tiny{\textsf{T}}}}}

\theoremstyle{definition}
\newtheorem{defn}{Definition}[section]
\newtheorem{rmk}[defn]{Remark}

\theoremstyle{plain}
\newtheorem{lem}[defn]{Lemma}
\newtheorem{thm}[defn]{Theorem}
\newtheorem{cor}[defn]{Corollary}

\newtheorem{ass}[defn]{Assumption}
\newtheorem{prop}[defn]{Proposition}
\numberwithin{equation}{section}

\begin{document}
\begin{frontmatter}

\title{Quadratic Semimartingale BSDEs under an Exponential Moments Condition}
\runtitle{Quadratic BSDEs under exponential moments}

\author{\fnms{Markus} \snm{Mocha\thanksref{t2}\hspace{-0.3cm}}\ead[label=e2]{mocha@math.hu-berlin.de}}
\address{Markus Mocha\\ Institut f\"{u}r Mathematik\\
Humboldt-Universit\"{a}t zu Berlin \\ Unter den Linden 6, 10099 Berlin \\ Germany\\
\printead{e2}}
\thankstext{t2}{Corresponding author: \texttt{mocha@math.hu-berlin.de}}
\and
\author{\fnms{Nicholas} \snm{Westray}\ead[label=e4]{westray@math.hu-berlin.de}}
\address{Nicholas Westray\\ Institut f\"{u}r Mathematik\\ Humboldt-Universit\"{a}t zu Berlin \\ Unter den Linden 6, 10099 Berlin \\ Germany\\
\printead{e4}}

\affiliation{Humboldt-Universit\"{a}t zu Berlin}

\runauthor{M. Mocha and N. Westray}


\begin{abstract}
In the present article we provide existence, uniqueness and stability results under an exponential moments condition for quadratic semimartingale backward stochastic differential equations (BSDEs) having convex generators. We show that the martingale part of the BSDE solution defines a true change of measure and provide an example which demonstrates that pointwise convergence of the drivers is not sufficient to guarantee a stability result within our framework.
\end{abstract}

\end{frontmatter}


\section{Introduction}
Since their introduction by Bismut \cite{Bi73} within the Pontryagin maximum principle, backward stochastic differential equations (BSDEs) have attracted much attention in the mathematical literature. In a Brownian framework such equations are usually written
\begin{equation}\label{BSDEW}
dY_t=Z_t\,dW_t-F(t,Y_t,Z_t)\,dt, \quad Y_T=\xi,
\end{equation}
where $\xi$ is an $\mc{F}_T$-measurable random variable, the \emph{terminal value}, and $F$ is the so called \emph{driver} or \emph{generator}. Here $(\mc{F}_t)_{t\in[0,T]}$ denotes the filtration generated by the one dimensional Brownian motion $W$. Solving such a BSDE corresponds to finding a \emph{pair} of adapted processes $(Y,Z)$ such that the integrated version of \eqref{BSDEW} holds. The presence of the \emph{control process} $Z$ stems from the requirement of adaptedness for $Y$ together with the fact that $Y$ must be driven into the random variable $\xi$ at time $T$. One may think of $Z$ as arising from the martingale representation theorem.

In the general semimartingale framework, where the main source of randomness is encoded in a given local martingale $M$ on a filtration $(\mc{F}_t)_{t\in[0,T]}$ that is not necessarily generated by $M$, we have to add an extra orthogonal component $N$. The corresponding BSDE then takes the form
\begin{equation}\label{BSDEM}
dY_t=Z_t\,dM_t+dN_t-f(t,Y_t,Z_t)\,d\lo M\ro_t-g_t\,d\lo N\ro_t, \quad Y_T=\xi.
\end{equation}
Solving \eqref{BSDEM} now corresponds to finding an adapted \emph{triple} $(Y,Z,N)$ of processes satisfying the integrated version of \eqref{BSDEM}, where $N$ is a local martingale orthogonal to $M$. We refer to $Z\cdot M+N$ as the \emph{martingale part} of a solution to the BSDE \eqref{BSDEM}.

BSDEs of type \eqref{BSDEW} and \eqref{BSDEM} have found many fields of application in mathematical finance. The first problems to be attacked by means of such equations included pricing and hedging, superreplication and recursive utility. The reader is directed to the survey articles El Karoui, Peng and Quenez \cite{EKPQ97} and El Karoui, Hamad{\`e}ne and Matoussi \cite{EKHM09} and the references therein for further discussion. A second large focus has been on their use in constrained utility maximization. In a Brownian setting Hu, Imkeller and M{\"u}ller \cite{HIM05} used the martingale optimality principle to derive a BSDE for the value process characterizing the optimal wealth and investment strategy. Their article can be regarded as an extension of earlier work by Rouge
and El Karoui \cite{REK00} as well as Sekine \cite{Se06}. In related work in a semimartingale setting Mania and Schweizer \cite{MS05} used a BSDE to describe the dynamic indifference
price for exponential utility. Their stochastic control approach was extended to robust utility in Bordigoni, Matoussi and Schweizer \cite{BMS07} and to an infinite time horizon in the recent article by Hu and Schweizer \cite{HS09}. We also mention Becherer \cite{Be06} for further extensions to BSDEs with jumps and Mania and Tevzadze \cite{MT08} to backward stochastic partial differential equations. This list is by far not exhaustive and
additional references can be found in the stated papers.

With regards to the theory of BSDEs, existence and uniqueness results were first provided in a Brownian setting by Pardoux and Peng \cite{PP90} under
Lipschitz conditions. These were extended by Lepeltier and San Mart\'{i}n \cite{LSM97} to continuous drivers with linear growth and by Kobylanski \cite{K00} to generators which are quadratic as a function of the control variable $Z$. The corresponding results for the semimartingale case may be found in Morlais \cite{Mo09} and Tevzadze \cite{Te08}, where in the former the main theorems of \cite{HIM05} are extended. In addition a stability result for quadratic BSDEs may also be found in the recent article by Frei \cite{Fr09}. In the situation when the generator has superquadratic growth, Delbaen, Hu and Bao \cite{DHB10} show that such BSDEs are essentially ill-posed.

A strong requirement present in the articles \cite{K00,Mo09,Te08} is that the terminal condition be bounded. In a Brownian setting Briand and Hu \cite{BH06,BH08} have replaced this by the assumption that it need only have exponential moments but in addition the driver is convex in the $Z$ variable. More recently, by interpreting the $Y$ component as the solution to a stochastic control problem, Delbaen, Hu and Richou \cite{DHR09} extend their results and show that one can reduce the order of exponential moments required.

The present article has two main contributions, the first is to extend the existence, uniqueness and stability theorems of \cite{BH08} and \cite{Mo09} to the unbounded semimartingale case. The motivation here is predominantly mathematical, having results in greater generality increases the range of applications for BSDEs. We remark however, that there are additional practical applications for the results derived here, e.g. related to utility maximization with an unbounded mean variance tradeoff, see Nutz \cite{Nu209} and Mocha and Westray \cite{MW10a}, which provides a second motivation for the present work.

In order to prove the respective results in the unbounded semimartingale framework technical difficulties related to an a priori estimate must be overcome. This requires an additional assumption when compared to \cite{BH08} and \cite{Mo09}. As a biproduct of establishing our results we are able to show via an example that the stability theorem as stated in \cite{BH08} Proposition 7 needs a minor amendment to the mode of convergence assumed on the drivers and we include the appropriate formulation. 

Our second contribution is to address the question of measure change. It is a classical result that when the generator has quadratic growth, the solution processes $Y$ is bounded if and only if the martingale part $Z\cdot M+N$ is a BMO martingale. In the present setting such a correspondence is lost, however we are able to show that whilst $Z\cdot M+N$ need not be BMO, see Frei, Mocha and Westray \cite{FMW11} for further discussion and some examples, the stochastic exponential $\mc{E}\big(q(Z\cdot M+N)\big)$ is still a true martingale for $q$ valid in some half-line. It is not only mathematically interesting to be able to describe the properties of the martingale part of the BSDE but also relevant for applications in an unbounded setting. For instance, the above result has been used in Heyne \cite{He10} to extend the results of \cite{HIM05} and \cite{Mo09} on utility maximization. Moreover such a theorem may be used in the partial equilibrium framework of Horst, Pirvu and dos Reis \cite{HPdR10} where the market price of (external) risk is given by equilibrium considerations and is typically unbounded. 

The paper is organized as follows. In the next section we lay out
the notation and the assumptions and state the main
results. The subsequent sections contain the proofs. Section \ref{SecAPriori} gives the a priori estimates together with some remarks on the necessity of an additional assumption, Section \ref{Exist} deals with existence and Section \ref{Uniq}
includes the comparison and uniqueness results. In Sections 6 and 7 we
prove the stability property as well as providing an interesting counterexample. In Section 8, we turn our attention to
the measure change problem and finally, in Section 9, we give interesting applications of our results to constrained utility maximization and partial equilibrium models.


\section{Model Formulation and Statement of Results}\label{SecModForm}
We work on a filtered probability space $(\Omega,\mc{F},(\mc{F}_t)_{0\leq t\leq T},\mb{P})$ satisfying the usual conditions of right-continuity and completeness. We also assume that $\mc{F}_0$ is the completion of the trivial $\sigma$-algebra. The time horizon $T$ is a finite number in $(0,\infty)$ and all semimartingales are considered equal to their c{\`a}dl{\`a}g modification.

Throughout this paper $M=(M^1,\ldots,M^d)^\tr$ stands for a continuous $d$-dimensional local martingale, where ${}^\tr$ denotes transposition. We refer the reader to Jacod and Shiryaev \cite{JS03} and Protter \cite{Pr05} for further details on the general theory of stochastic integration.

The objects of study in the present paper will be semimartingale BSDEs considered on $[0,T]$. In the $d$-dimensional case such a BSDE may be written
\begin{gather}
dY_t=Z_t^{\tr}\,dM_t+dN_t-\textbf{1}^\tr\,d\lo M\ro_tf(t,Y_t,Z_t)-g_t\,d\lo N\ro_t,\quad Y_T=\xi\label{BSDEf}.
\end{gather}
Here $\xi$ is an $\R$-valued $\mc{F}_T$-measurable random variable and $f$ and $g$ are random predictable functions $[0,T]\times\Omega\times\mb{R}\times\mb{R}^d\to\mb{R}^d$ and $[0,T]\times\Omega\to\mb{R}$, respectively. We set $\textbf{1}:=(1,\ldots,1)^{\tr}\in\mb{R}^d$.

The format in which the BSDE \eqref{BSDEf} encodes its finite variation parts is not so tractable from the point of view of analysis. Therefore we write semimartingale BSDEs by factorizing the matrix-valued process $\lo M\ro=\lo M^i,M^j\ro_{i,j=1,\ldots,d}$. This separates its matrix property from its nature as measure. This step could also be regarded as a reduction of dimensionality.

For $i,j\in\{1,\ldots,d\}$ we may write $\lo M^i,M^j\ro=C^{ij}\cdot A$ where $C^{ij}$ are the components of a predictable process $C$ valued in the space of symmetric positive semidefinite $d\times d$ matrices and $A$ is a predictable increasing process. There are many such factorizations (cf. \cite{JS03} Section III.4a). We may choose $A:=\arctan\!\left(\sum_{i=1}^d\big\lo M^i\big\ro\right)$ so that $A$ is uniformly bounded by $K_A=\pi/2$ and derive the absolute continuity of all the $\lo M^i,M^j\ro$ with respect to $A$ from the Kunita-Watanabe inequality. This together with the Radon-Nikod{\'y}m theorem provides $C$. Furthermore, we can factorize $C$ as $C=B^{\tr}B$ for a predictable process $B$ valued in the space of $d\times d$ matrices. We note that all the results below do not rely on the specific choice of $A$, but only on its \emph{boundedness}. In particular, if $M=W$ is a $d$-dimensional Brownian motion we may choose $A_t=t$, $t\in[0,T]$, and $B$ the identity matrix. Then $A$ is bounded by $K_A=T.$

We let $\mc{P}$ denote the predictable $\sigma$-algebra on $[0,T]\times \Omega$ generated by all the left-continuous processes. The process $A$ induces a measure $\mu^A$ on $\mc{P}$, the \emph{Dol{\'e}ans measure}, defined for $E\in\mc{P}$ by \[\mu^A(E):=\E\!\left[\int_0^T\I{E}(t)\,dA_t\right].\]

Given the above discussion the equation \eqref{BSDEf} may be rewritten as
\begin{gather}
dY_t=Z_t^{\tr}\,dM_t+dN_t-F(t,Y_t,Z_t)\,dA_t-g_t\,d\lo N\ro_t,\label{BSDEF-}\quad Y_T=\xi,
\end{gather}
where again $\xi$ is an $\R$-valued $\mc{F}_T$-measurable random variable, the \emph{terminal condition}, and $F$ and $g$ are random predictable functions $[0,T]\times\Omega\times\mb{R}\times\mb{R}^d\to\mb{R}$ and $[0,T]\times\Omega\to\mb{R}$ respectively, called \emph{generators} or \emph{drivers}. This formulation of the BSDE is very flexible, allowing for various applications and being amenable to analysis.

Starting with \eqref{BSDEf} and setting $F(t,y,z):=\textbf{1}^{\tr}C_tf(t,y,z)=\textbf{1}^{\tr}B^{\tr}_tB_tf(t,y,z)$ we get \eqref{BSDEF-}. The reversion of this procedure is not so clear, however is not relevant in applications.

Under boundedness assumptions, existence of solutions to
\eqref{BSDEF-} is provided in \cite{Mo09} via an exponential
transformation that makes the $d\lo N\ro$ term disappear. A
necessary condition for this kind of transformation to work
properly is $dg=0$. In the sequel we thus consider the above BSDE
to be given in the form
\begin{gather}
dY_t=Z_t^{\tr}\,dM_t+dN_t-F(t,Y_t,Z_t)\,dA_t-\frac{1}{2}\,d\lo N\ro_t,\quad \label{BSDEF} Y_T=\xi,
\end{gather}
except in specific situations where a solution is assumed to exist.
\begin{defn}
A \emph{solution} to the BSDE \eqref{BSDEF-}, or \eqref{BSDEF}, is a
triple $(Y,Z,N)$ of processes valued in $\R\times\R^d\times\R$
satisfying \eqref{BSDEF-}, or \eqref{BSDEF}, $\mb{P}$-a.s. such that:
\begin{enumerate}
\item The function $t\mapsto Y_t$ is continuous $\mb{P}$-a.s.
\item The process $Z$ is predictable and $M$-integrable, in particular $\int_0^T Z^\tr_t\,d\lo M\ro_tZ_t<+\infty$ $\mathbb{P}$-a.s.
\item The local martingale $N$ is continuous and orthogonal to each component of $M$, i.e. $\lo N,M^i \ro =0$ for all $i=1,\ldots,d$.
\item We have that $\mb{P}$-a.s.\[\int_0^T|F(t,Y_t,Z_t)|\,dA_t+\lo N\ro_T<+\infty.\]
\end{enumerate}
As in the introduction we call $Z\cdot M+N$ the \emph{martingale part} of a solution. 
\end{defn}

In what follows we collect together the assumptions that allow for all the assertions of this paper to hold simultaneously. However we want to point out that not all of our results require that every item of Assumption \ref{ass} be satisfied, as will be indicated in appropriate remarks.

\begin{ass}\label{ass}
There exist nonnegative constants $\beta$ and $\ol{\beta}$, positive numbers $\beta_f$ and $\gamma\geq \max(1,\beta)$ together with an $M$-integrable (predictable) $\mb{R}^d$-valued process $\lambda$ so that writing
\begin{equation*}
\a:=\|B\lambda\|^2 \text{ and } \,|\a|_1:=\int_0^T\a_t\,dA_t=\int_0^T\lambda^\tr_t\,d\lo M\ro_t\lambda_t
\end{equation*}
we have $\mathbb{P}$-a.s.
\begin{enumerate}
\item The random variable $|\xi|+|\a|_1$ has exponential moments of all orders, i.e.
for all $p>1$
\begin{equation}
\E\!\left[\exp\!\Big(p\,\big[|\xi|+|\a|_1\big]\Big)\right]<+\infty.\label{AssExpMom}
\end{equation}
\item For all $t\in[0,T]$ the driver $(y,z)\mapsto F(t,y,z)$ is continuous in $(y,z)$, convex in $z$ and Lipschitz continuous in $y$ with Lipschitz constant $\ol{\beta}$, i.e. for all $y_1$, $y_2$ and $z$ we have
\begin{equation}
|F(t,y_1,z)-F(t,y_2,z)|\leq \ol{\beta}\,|y_1-y_2|.\label{AssLipF}
\end{equation}
\item The generator $F$ satisfies a quadratic growth condition in $z$, i.e. for all $t, y$ and $z$ we have
\begin{equation}
|F(t,y,z)|\leq \alpha_t+\a_t \beta|y|+\frac{\gamma}{2}\|B_tz\|^2.\label{AssGrowF}
\end{equation}
\item The function $F$ is locally Lipschitz in $z$, i.e. for all $t, y, z_1$ and $z_2$
\begin{equation*}
|F(t,y,z_1)-F(t,y,z_2)|\leq \beta_f \Big(\|B_t\lambda_t\| +\|B_tz_1\|+\|B_tz_2\|\Big)\|B_t(z_1-z_2)\|.\label{AssLocLipF}
\end{equation*}
\item The constant $\beta$ in (iii) equals zero and then we set $c_A:=0.$ Alternatively, $\beta>0$, but additionally assume that for all $y$ and $z$ we have
\begin{equation*}
|F(t,y,z)-F(t,0,z)|\leq \ol{\beta}\,|y|
\end{equation*}
and that there is a positive constant $c_A$ such that $A_t\leq c_A\cdot t$ for all $t\in[0,T]$.
\end{enumerate}
If this assumption is satisfied we refer to \eqref{BSDEF} as BSDE$(F,\xi)$ with the set of parameters $(\a,\beta,\ol{\beta},\beta_f,\gamma)$. 
\end{ass}

\begin{rmk}
The above items (i)-(iv) correspond to the assumptions made in \cite{BH08} and \cite{Mo09}. In particular, the BSDEs under consideration are of \emph{quadratic type} (in the control variable $z$) and of Lipschitz type in $y$. Item (v) is new and arises from the fact that the methods used in \cite{Mo09} to derive an a priori estimate may no longer be directly applied so that an additional assumption is required. We elaborate further on this topic in Section \ref{SecAPriori}. Observe that in the key application of utility maximization, cf. \cite{MW10a}, the associated driver is independent of $y$ and hence $\beta=0$ applies. 
\end{rmk}

Notice that items (ii) and (iii) from above provide
\begin{equation}
|F(t,y,z)|\leq \a_t+\ol{\beta}|y|+\frac{\gamma}{2}\,\|B_tz\|^2,\label{Ftyzbetas}
\end{equation}
for all $t$, $y$ and $z$, $\mb{P}$-a.s. This is an inequality which does not involve $\a$ in the $|y|$ term on the right hand side and which is used repeatedly throughout the proofs. We also define the constant
\[\beta^*:=c_A\cdot\ol{\beta}.\]

Before giving the main results of the paper let us introduce some notation. For $p\geq1$, $\mc{S}^p$ denotes the set of $\R$-valued, adapted and continuous processes $Y$ on $[0,T]$ such that \[\E\!\left[\sup_{\substack{0\leq t\leq T} }|Y_t|^p\right]^{1/p}<+\infty.\] The space $\mc{S}^\infty$ consists of the continuous bounded processes. An $\R$-valued, adapted and continuous process $Y$ belongs to $\mathfrak{E}$ if the random variable \[Y^*:=\sup_{\substack{t\in[0,T]}}|Y_t|\] has exponential moments of all orders. We also recall that $Y$ is called \emph{of class D} if the family $\{Y_\tau|\,\tau\in[0,T]\text{ stopping time}\}$ is uniformly integrable. The set of (equivalence classes of) $\mb{R}^d$-valued predictable processes $Z$ on $[0,T]\times \Omega$ satisfying
\[\E\!\left[\!\left(\int_0^TZ_t^{\tr}\,d\lo M\ro_tZ_t\right)^{p/2}\right]^{1/p}<+\infty\]
is denoted by $\mathfrak{M}^p$. Finally, $\mc{M}^p$ stands for the set of $\R$-valued martingales $N$ on $[0,T]$, such that
\[\|N\|_{\mc{M}^p}:=\E\!\left[\lo N\ro_T^{p/2}\right]^{1/p}<+\infty.\]
Notice that if the following assumption on the filtration holds the elements of $\mc{M}^p$ are continuous.
\begin{ass}\label{ass_filtr}
The filtration $(\mc{F}_t)_{t\in[0,T]}$ is a continuous filtration, in the sense that all local $(\mc{F}_t)_{t\in[0,T]}$-martingales are continuous.
\end{ass}

\noindent The following four theorems constitute the main results of the paper. We mention that only the existence result requires the assumption of the continuity of the filtration. 

\begin{thm}[Existence]\label{ThmExist}
If Assumptions \ref{ass} and \ref{ass_filtr} hold there exists a solution $(Y,Z,N)$ to the BSDE \eqref{BSDEF} such that $Y\in\mathfrak{E}$ and $Z\cdot M+N\in\mc{M}^p$ for all $p\geq1$.
\end{thm}

\begin{thm}[Uniqueness]\label{ThmUniq}
Suppose that Assumption \ref{ass} holds. Then any two solutions $(Y,Z,N)$ and $(Y',Z',N')$ in $\mathfrak{E}\times\mathfrak{M}^2\times\mc{M}^2$ to the BSDE \eqref{BSDEF} coincide in the sense that $Y$ and $Y'$, $Z\cdot M$ and $Z'\cdot M$, and $N$ and $N'$ are indistinguishable.
\end{thm}

\begin{thm}[Stability]\label{ThmStab}
Consider a family of BSDEs($F^n,\xi^n$) indexed by the extended natural numbers $n\geq0$ for which Assumption
\ref{ass} holds true with parameters $(\a^n,\beta^n,\ol{\beta},\beta_f,\gamma)$. Assume that the
exponential moments assumption \eqref{AssExpMom} holds uniformly
in $n$, i.e. for all $p>1$,
\begin{equation*}
\sup_{\substack{n\geq 0}}\E\!\left[e^{p\,(|\xi^n|+|\a^n|_1)}\right]<+\infty.
\end{equation*}
If for $n\geq0$ $(Y^n,Z^n,N^n)$ is the solution in $\mathfrak{E}\times\mathfrak{M}^2\times\mc{M}^2$ to the BSDE($F^n,\xi^n$) and if
\begin{equation}
|\xi^n-\xi^0| + \int_0^T\big|F^n-F^0\big|\,(s,Y_s^0,Z_s^0)\,dA_s\longrightarrow 0 \quad\text{ in
probability, as }n\to+\infty,
\end{equation}
then for each $p\geq1$ as $n\to+\infty$
\begin{gather*}
\E\Bigg[\!\!\left(\exp\!\left(\sup_{\substack{0\leq t\leq T}}\big|Y_t^n-Y_t^0\big|\right)\!\right)^p\Bigg]\longrightarrow 1\quad\text{and}\quad Z^n\cdot M+N^n\longrightarrow Z^0\cdot M+N^0 \text{ in } \mc{M}^p.
\end{gather*}
\end{thm}

\begin{thm}[Exponential Martingales]\label{TrueMart}
Suppose that Assumption \ref{ass} holds, let $|q|>\gamma/2$ and let $(Y,Z,N)\in\mathfrak{E}\times\mathfrak{M}^2\times\mc{M}^2$ be a solution to the BSDE \eqref{BSDEF}. Then $\mc{E}\bigl(q\,(Z\cdot M+N)\bigr)$ is a true martingale on $[0,T]$.
\end{thm}

\begin{rmk}
The preceding theorems generalize the results of \cite{BH08} and \cite{Mo09} and the method of proof is therefore similar. We combine the localization and $\theta$-technique from \cite{BH08} together with the existence and stability results for BSDEs with bounded solutions found in \cite{Mo09}. Similar ideas are used in \cite{HS09} on a specific quadratic BSDE arising in a robust utility maximization problem where the authors also investigate the measure change problem for their special BSDE, however here we pursue the general theory. We point out that when the BSDE is of quadratic type and $|\xi|+|\alpha|_1$ does not have sufficiently large exponential moments there are examples where the BSDE admits no solution. Thus the results here can be considered, in some sense, the best possible. In particular, we present all the theoretical background for the study of utility maximization under exponential moments, see \cite{He10} and \cite{MW10a}, as well as partial equilibrium, see \cite{HPdR10}.
\end{rmk}


\section{A Priori Estimates}\label{SecAPriori}
In this section we show that, under appropriate
conditions, solutions to the BSDE \eqref{BSDEF-} satisfy some a
priori norm bounds. After giving an important result used in the subsequent sections we motivate Assumption \ref{ass} (v) by showing that without such an assumption the method utilized in \cite{Mo09} for the purpose of deriving appropriate a priori bounds fails in the present unbounded case.

Let $(Y,Z,N)$ be a solution to \eqref{BSDEF-}, suppose that Assumption \ref{ass} (iii) and (v)
hold and that $g$ is uniformly
bounded by $\gamma/2$. Fix $s\in[0,T]$ and set, for $t\in[s,T]$,
\begin{equation*}
\ti{H}_t:=\exp\!\left(\gamma e^{\beta^*(t-s)}|Y_t|+\gamma\int_s^te^{\beta^*(r-s)}\,d\lo\lambda\cdot M\ro_r\right).
\end{equation*}
where we have written $\lo \lambda\cdot M\ro_t:=\int_0^t\lambda_r^\tr\,d\lo M\ro_r\lambda_r=\int_0^t\a_r\,dA_r$. First we show that $\ti{H}$ is, up to integrability, a local submartingale.

From Tanaka's formula,
\begin{multline}
d|Y_t|=\sgn(Y_t)(Z_t^{\tr}\,dM_t+dN_t)-\sgn(Y_t)\big(F(t,Y_t,Z_t)\,dA_t+g_t\,d\lo N\ro_t\big)+dL_t\label{dabsY},
\end{multline}
where $L$ is the local time of $Y$ at 0. It{\^o}'s formula then yields
\begin{align*}
d\ti{H}_t=\gamma \ti{H}_t\,&e^{\beta^*(t-s)}\Bigg[\sgn(Y_t)(Z_t^{\tr}\,dM_t+dN_t) +\ol{\beta}|Y_t|({c}_A\,dt-dA_t)\\
&+\left(-\sgn(Y_t)F(t,Y_t,Z_t)+\a_t+\ol{\beta}|Y_t|+\frac{\gamma}{2}\,e^{\beta^*(t-s)}\|B_tZ_t\|^2\right)dA_t\\
&+\left(-\sgn(Y_t)\,g_t
+\frac{\gamma}{2}\,e^{\beta^*(t-s)}\right)d\lo N\ro_t
+dL_t\Bigg].
\end{align*}
An inspection of the finite variation parts shows that under the present assumptions they are nonnegative. In particular, the semimartingale $\ti{H}$ is a local submartingale, 
which leads to the following result.
\begin{prop}[A Priori Estimate]\label{apriori0}
Suppose Assumption \ref{ass} (iii) and (v) hold and assume that the function $g$ is uniformly bounded by $\gamma/2$, $\mb{P}$-a.s. Let
$(Y,Z,N)$ be a solution to the BSDE \eqref{BSDEF-} and let the process
\begin{equation*}
\exp\!\left(\gamma e^{\beta^*T}|Y|+\gamma\int_0^Te^{\beta^*r}\,d\lo\lambda\cdot M\ro_r \right)
\end{equation*}
be of class D. Then $\mb{P}$-a.s. for all $s\in[0,T]$,
\begin{equation}
|Y_s|\leq\frac{1}{\gamma}\log\E\!\left[\exp\!\left(\gamma e^{\beta^*(T-s)}|\xi|+\gamma\int_s^Te^{\beta^*(r-s)}\,d\lo\lambda\cdot M\ro_r \right)\!\Bigg|\,\mc{F}_s\right]\label{aposteriori}.
\end{equation}
\end{prop}
\begin{proof}
Fix $s\in[0,T]$ and set $\ti{H}$ as above. Since $\ti{H}$ is a local submartingale there exists a sequence of stopping times $(\tau_n)_{n\geq1}$ valued in $[s,T]$, which converges $\mb{P}$-a.s. to $T$, such that $\ti{H}^{\tau_n}$ is a submartingale for each $n\geq 1.$ 
We then derive
\begin{align*}
\exp(\gamma |Y_s|)
\leq \E[\ti{H}_{T\wedge\tau_n}|\,\mc{F}_s]\leq
\E\!\left[\exp\!\left(\gamma  e^{\beta^*(T-s)}|Y_{T\wedge\tau_n}|+\gamma\int_s^Te^{\beta^*(r-s)}\,d\lo \lambda \cdot
M\ro_r\right)\!\Bigg|\,\mc{F}_s\right].
\end{align*}
Letting $n\to+\infty$ the claim follows from the class D
assumption.
\end{proof}
Proposition \ref{apriori0} provides the appropriate a priori
estimate, indeed suppose that $|\xi|$ and $|\a|_1$ are bounded random variables and
$(Y,Z,N)$ is a solution to \eqref{BSDEF}. If the current
assumptions hold and $\exp(\gamma e^{\beta^*T}|Y|)$ is
of class D, then $Y$ satisfies
\begin{equation}\label{a1bound}
|Y|\leq\left\| e^{\beta^*T}(|\xi|+|\a|_1)\right\|_\infty.
\end{equation}
Comparing with \eqref{aposteriori} this indicates that the inclusion of Assumption \ref{ass} (v) allows us to prove similar estimates to the bounded case which enables us to establish existence for the BSDE \eqref{BSDEF} when $|\xi|+|\a|_1$ has exponential moments of all orders, to be more precise, an order of at least $\gamma
e^{\beta^*T}$.

Contrary to the above let us investigate the method utilized in \cite{Mo09} under Assumption \ref{ass} (iii) only, supposing that $g$ be bounded by $\gamma/2$. We set
\begin{equation}
H_t:=\exp\!\left(\gamma e^{\beta\lo \lambda \cdot M\ro_{s,t}}|Y_t|+\gamma\int_s^te^{\beta\lo \lambda\cdot M\ro_{s,r}}\,d\lo\lambda\cdot M\ro_r\right),
\end{equation}
where $\lo \lambda \cdot M\ro_{s,t}:=\lo \lambda \cdot M\ro_{t}-\lo \lambda \cdot M\ro_{s}=\int_s^t\a_r\,dA_r$. We derive from It{\^o}'s formula
\begin{align*}
dH_t=
\gamma H_t\,&e^{\beta\lo\lambda\cdot M\ro_{s,t}}\Bigg[\sgn(Y_t)(Z_t^{\tr}\,dM_t+dN_t) \\
&+\left(-\sgn(Y_t)F(t,Y_t,Z_t)+\a_t+\a_t\beta|Y_t|+\frac{\gamma}{2}\,e^{\beta\lo\lambda\cdot M\ro_{s,t}}\|B_tZ_t\|^2\right)dA_t\\
&+\left(-\sgn(Y_t)\,g_t +\frac{\gamma}{2}\,e^{\beta\lo\lambda\cdot M\ro_{s,t}}\right)d\lo N\ro_t +dL_t\Bigg].
\end{align*}
Once again, the finite variation parts are nonnegative. We conclude in the same way as for Proposition \ref{apriori0} that the corresponding a priori result holds for $H$ as well. To sum up, we have that under a similar class D assumption, now on
\begin{equation*}
\exp\!\left(\gamma e^{\beta\lo \lambda \cdot M\ro_T}|Y|+\gamma\int_0^Te^{\beta\lo \lambda\cdot M\ro_{r}}\,d\lo\lambda\cdot M\ro_r\right),
\end{equation*}
$\mb{P}$-a.s. for all $s\in[0,T]$,
\begin{equation}
|Y_s|\leq\frac{1}{\gamma}\log\E\!\left[\exp\left(\gamma e^{\beta \lo \lambda \cdot M\ro_{s,T}}|\xi|+\gamma\int_s^Te^{\beta \lo \lambda \cdot M\ro_{s,r}}\,d\lo\lambda\cdot M\ro_r \right)\!\Bigg|\,\mc{F}_s\right].\label{eqapriori}
\end{equation}
If $\beta=0$, then $\ti{H}$ from above equals $H$ and there is no difference with the statement of Proposition \ref{apriori0}. However when $\beta>0$ the estimate \eqref{eqapriori} is not sufficient for our purposes. We aim at using the a priori estimate to show the existence of solutions to the BSDE \eqref{BSDEF} in $\mathfrak{E}\times\mathfrak{M}^2\times\mc{M}^2$
using an appropriate approximating procedure. If $|\xi|$ and $|\a|_1$ are bounded random variables there exists a
solution $(Y,Z,N)$ to \eqref{BSDEF} with $Y$ bounded, cf.
\cite{Mo09}. With \eqref{eqapriori} at our disposal we then have the estimate 
\begin{equation}\label{a1boundpre}
|Y|\leq\left\|e^{\beta
|\a|_1}\big(|\xi|+|\a|_1\big)\right\|_\infty.
\end{equation}
Our goal is to remove the boundedness assumption and to replace it
with the assumption on the existence of exponential moments of
$|\xi|+|\a|_1$ in the spirit of \cite{BH08}. However a closer
inspection of the a priori estimate from \eqref{eqapriori} together with \eqref{a1boundpre}
already indicates that more restrictive assumptions are necessary. More specifically, when $\beta>0$ we cannot deduce any integrability
of $\exp\!\left(\gamma
e^{\beta|\a|_1}\big(|\xi|+|\a|_1\big)\right)$ when $|\xi|$ and
$|\a|_1$ have only exponential moments, this motivates Assumption \ref{ass} (v).


\section{Existence}\label{Exist}
In the present section we establish Theorem \ref{ThmExist} together with some related results on norm bounds of the solution.
The proof of existence follows the following recipe. Firstly we truncate $\lo \lambda\cdot M\ro$ to get approximate solutions. Then by using the estimate from Proposition \ref{apriori0} we localize and work on a random time interval so that the approximations are uniformly bounded and we can apply a stability result. Finally we glue together on $[0,T]$ to construct a solution. The a priori estimates ensure that we may take all limits in the described procedure.

\begin{thm}[Existence]\label{ThmExistPr}
Let Assumptions \ref{ass} (ii)-(v) and \ref{ass_filtr} hold and $|\xi|+|\a|_1$ have an exponential moment of order $\gamma
e^{\beta^*T}$. Then the BSDE \eqref{BSDEF} has a solution
$(Y,Z,N)$ such that
\begin{equation}\label{aposterioriPr}
|Y_t|\leq\frac{1}{\gamma}\log\E\!\left[\exp\left(\gamma\, e^{\beta^*(T-t)}|\xi|+\gamma\int_t^Te^{\beta^*(r-t)}\,d\lo\lambda\cdot M\ro_r \right)\!\Bigg|\,\mc{F}_t\right].
\end{equation}
\end{thm}

\begin{proof}
Exactly as in \cite{BH08} we first assume that $F$ and $\xi$ are
nonnegative. For each integer $n\geq 1$, set
\[ \sigma_n:=\inf\!\left\{t\in[0,T]\,\bigg|\lo\lambda\cdot M\ro_t:=\int_0^t\alpha_s\,dA_s\geq n\right\}\wedge T,\]
$\xi^n:=\xi\wedge n$, $\lambda^n_t:=\I{\!\{t\leq
\sigma_n\!\}}\lambda_t$ and $F^n(t,y,z):=\I{\!\{t\leq
\sigma_n\!\}}F(t,y,z)$. Then $F^n$ satisfies Assumption \ref{ass} (ii)-(v) with
the same constants, but with the processes $\lambda^n$ and $\a^n$
where
\[\a^n_t:=\|B_t\lambda_t^n\|^2=\I{\!\{t\leq \sigma_n\!\}}\|B_t\lambda_t\|^2=\I{\!\{t\leq \sigma_n\!\}}\a_t
.\] In particular, $|\a^n|_1=\int_0^{\sigma_n}\alpha_s\,dA_s\leq n$ and
\[\int_0^T(\lambda^n_t)^\tr\,d\lo M\ro_t\lambda^n_t=\int_0^T\|B_t\lambda^n_t\|^2\,dA_t=|\a^n|_1\leq n,\] so we may apply \cite{Mo09} Theorem 2.5 and Theorem 2.6 to conclude that there exists a unique solution $(Y^n,Z^n,N^n)\in\mc{S}^\infty\times \mathfrak{M}^2\times \mc{M}^2$ to the BSDE \eqref{BSDEF}, where $F$ is replaced by $F^n$ and $\xi$ by $\xi^n$. From Proposition \ref{apriori0} we derive
\begin{align}
|Y^n_t|&\leq\frac{1}{\gamma}\log\E\!\left[\exp\!\left(\gamma e^{\beta^*(T-t)}|\xi^n|+\gamma\int_t^Te^{\beta^*(r-t)}\,d\lo\lambda^n\cdot M\ro_r \right)\!\Bigg|\,\mc{F}_t\right]\notag\\
&\leq\frac{1}{\gamma}\log\E\!\left[\exp\!\left(\gamma e^{\beta^*(T-t)}|\xi|+\gamma\int_t^Te^{\beta^*(r-t)}\,d\lo\lambda\cdot M\ro_r \right)\!\Bigg|\,\mc{F}_t\right]\notag\\
&\leq\frac{1}{\gamma}\log\E\!\left[\exp\!\left(\gamma e^{\beta^*T}\big(|\xi|+|\a|_1\big)\right)\!\Bigg|\,\mc{F}_t\right]=:X_t.\label{Xt}
\end{align}
Let $n\leq m$ so that then we have $\sigma_n\leq \sigma_m$ and $\I{\{t\leq\sigma_n\!\}}\leq\I{\{t\leq \sigma_m\!\}}$. In particular, $\xi^n\leq\xi^m$ and $F^n\leq F^m$, from which we deduce that the current assumptions, hence the corresponding assumptions in \cite{Mo09}, hold for both $F^n$ and $F^m$ with the same set of parameters $(\a^m,\beta,\ol{\beta},\beta_f,\gamma)$ where the additional $c_\theta$ in \cite{Mo09} is equal to $m$. An application of Theorem 2.7 therein now shows that $Y^n\leq Y^m$ so that $(Y^n)_{n\geq 1}$ is an increasing sequence of bounded continuous processes.\\
Let $k\geq 1$ be a fixed integer and
\[ \tau_k:=\inf\!\left\{t\in[0,T]\,\Big|X_t\geq k \text{ or } \lo\lambda\cdot M\ro_t \geq k\right\}\wedge T.\]
Thanks to the continuity of the filtration the martingale $\exp(\gamma X)$ is continuous so that the random variable \[V:=\max_{t\in[0,T]}\bigl(X_t\bigr) \vee \lo \lambda\cdot M\ro_T\]
is finite $\mathbb{P}$-a.s. We derive that $\mathbb{P}$-a.s. $\tau_k=T$ for large $k$. Due to \eqref{Xt} the sequence $(Y^{n,k})_{n\geq 1}$ given by
\[Y^{n,k}_t:=Y^n_{t\wedge\tau_k},\]
is uniformly bounded by $k$. For the martingale parts we define
\[Z^{n,k}_t:=\I{\!\{t\leq \tau_k\!\}}Z^n_t\;\text{ and }\; N^{n,k}_t:=\I{\!\{t\leq \tau_k\!\}}N^n_t.\]
An inspection of the respective cases shows that
\begin{multline*}
Y^{n,k}_t=Y^n_{\tau_k}-\int_t^T\big(Z^{n,k}_s\big)^{\tr}\,dM_s-\int_t^TdN^{n,k}_s\\+\int_t^T\I{\!\{s\leq
\tau_k\wedge\sigma_n\!\}}F(s,Y^{n,k}_s,Z^{n,k}_s)\,dA_s+\frac{1}{2}\int_t^Td\lo
N^{n,k}\ro_s.
\end{multline*}
Moreover,
$Y^n_{\tau_k}\xrightarrow{n\uparrow+\infty}\sup_{\substack{n\geq
1}}Y_{\tau_k}^n=:\xi_k$, where $\xi_k$ is bounded by $k$. Next we
appeal to the stability result stated in \cite{Mo09} Lemma 3.3, noting
Remark 3.4 therein. Note that this result requires estimates
that are uniform in $n$ which is accomplished by the specific
choice of the stopping time $\tau_k$. Hence $(Y^{n,k},Z^{n,k},N^{n,k})$
converges to $(Y^{\infty,k},Z^{\infty,k},N^{\infty,k})$ in the sense that
\[\lim_{n\to+\infty}\E\!\left(\sup_{\substack{0\leq t \leq T}}\big|Y^{n,k}_t-Y^{\infty,k}_t\big|\right)=0,\]
\[\lim_{n\to+\infty}\E\!\left(\int_0^T\Big(Z^{n,k}_s-Z^{\infty,k}_s\Big)^{\tr}d\lo M\ro_s\Big(Z^{n,k}_s-Z^{\infty,k}_s\Big)\right)=0\]
and
\[\lim_{n\to+\infty}\E\!\left(\Big|N^{n,k}_T-N^{\infty,k}_T\Big|^2\right)=0,\]
where the triples $(Y^{\infty,k},Z^{\infty,k},N^{\infty,k})$ solve the BSDE
\begin{multline*}
dY^{\infty,k}_t=\Big(Z^{\infty,k}_t\Big)^{\tr}dM_t+dN^{\infty,k}_t\\-\I{\{t\leq\tau_k\!\}}F(t,Y^{\infty,k}_t,Z^{\infty,k}_t)\,dA_t-\frac{1}{2}\,d\lo
N^{\infty,k}\ro_t,\quad Y^{\infty,k}_{\tau_k}=\xi_k,
\end{multline*}
on the random horizon $[\![0,\tau_k]\!]\subset[0,T]$. 
The stopping times $\tau_k$ are monotone in $k$ and therefore it follows that
\[ Y^{n,k+1}_{t\wedge\tau_{k}}=Y^{n,k}_t,\quad \I{\!\{t\leq \tau_k\!\}}Z^{n,k+1}_t=Z^{n,k}_t \quad\text{and}\quad\I{\!\{t\leq \tau_k\!\}}N^{n,k+1}_t=N^{n,k}_t,\]
so that the above convergence yields (for the two last objects in $\mc{M}^2$)
\begin{gather*}
Y^{\infty,k+1}_{t\wedge\tau_k}=Y^{\infty,k}_t,\; \,\left(\!\left(\I{\{t\leq \tau_k\}}Z^{\infty,k+1}\right)\cdot M\right)_t=\left(Z^{\infty,k}\cdot
M\right)_t\text{ and }\, \I{\!\{t\leq
\tau_k\!\}}N^{\infty,k+1}_t=N^{\infty,k}_t.
\end{gather*}
To finish the proof, we define the processes
\begin{eqnarray*}
Y_t&:=&\I{\!\{t\leq \tau_1\!\}}Y^{\infty,1}_t+\sum_{k\geq 2}\I{\!\{t\in]\!]\tau_{k-1},\tau_k]\!]\}}Y^{\infty,k}_t,\\
Z_t&:=&\I{\!\{t\leq \tau_1\!\}}Z^{\infty,1}_t+\sum_{k\geq 2}\I{\!\{t\in]\!]\tau_{k-1},\tau_k]\!]\}}Z^{\infty,k}_t\\
\text{and }\,N_t&:=&\I{\!\{t\leq \tau_1\!\}}N^{\infty,1}_t+\sum_{k\geq 2}\I{\!\{t\in]\!]\tau_{k-1},\tau_k]\!]\}}N^{\infty,k}_t.
\end{eqnarray*}
By construction this gives a solution to the BSDE
\begin{align*}
dY_t=Z_t^{\tr}\,dM_t+dN_t-F(t,Y_t,Z_t)\,dA_t-\frac{1}{2}\,d\lo
N\ro_t,\quad Y_T=\xi,
\end{align*}
since $\I{\!\{t\leq \tau_1\!\}}+\sum_{k\geq 2}\I{\!\{t\in]\!]\tau_{k-1},\tau_k]\!]\}}=\I{\!\{t\in[0,T]\!\}}$ $\mathbb{P}$-a.s. More precisely, there is a $\mathbb{P}$-null set $\mathfrak{N}$ such that for all $\omega\in\mathfrak{N}^c$ there is a $k_0(\omega)$ with $\tau_{k_0(\omega)}(\omega)=T$ and such that $Y^{\infty,k}_{\tau_k(\omega)}(\omega)=\xi_k(\omega)$ for all $k$, which yields that (possibly after another modification of $\mathfrak{N}$) \[Y_T(\omega)=Y^{\infty,k_0(\omega)}_{T}(\omega)=\xi_{k_0(\omega)}(\omega)=\sup_{\substack{n\geq
1}}Y_{T}^n(\omega)=\xi(\omega).\]
The bound in \eqref{aposterioriPr} holds as we have it for all $n$ and $k$ from \eqref{Xt}.\\
In the case when $\xi$ and $f$ are not necessarily nonnegative, we proceed as in \cite{BH08} by using a double truncation defined by $\xi^{n,m}:=\xi^+\wedge n-\xi^-\wedge
m$, $\lambda^{n,m}:=\I{\!\{t\leq
\sigma_n\!\}}\lambda^+-\I{\!\{t\leq \sigma_m\!\}}\lambda^-$ and
$F^{n,m}:=\I{\!\{t\leq \sigma_n\!\}}F^+-\I{\!\{t\leq
\sigma_m\!\}}F^-$.
\end{proof}

As an immediate corollary we deduce
\begin{cor}[Norm Bounds]\label{Cornormbds}\mbox{}
\begin{enumerate}
\item Let Assumptions \ref{ass} (ii)-(v) and \ref{ass_filtr} hold and $|\xi|+|\a|_1$ have an exponential moment of order $\delta^*>\gamma e^{\beta^*T}$. Then the BSDE \eqref{BSDEF} has a solution $(Y,Z,N)$ such that $e^{\gamma Y}\in\mc{S}^{p^*}$ for $p^*:=\frac{\delta^*}{\gamma e^{\beta^* T}}>1$.\\
When additionally $|\xi|+|\a|_1$ has exponential moments of all orders, i.e.
Assumption \ref{ass} (i) holds, this solution is such that
$Y\in\mathfrak{E}$. In particular, for each $p>1$ we have the estimate
\begin{gather}
\E\!\left[e^{p\gamma
Y^*}\right]\leq\left(\frac{p}{p-1}\right)^p\E\!\left[\exp\!\bigg(p\gamma
e^{\beta^*T}\Big(|\xi|+|\a|_1\!\Big)\!\bigg)\right]\!.\label{normbdsY}
\end{gather}
\item Let Assumption \ref{ass} (i)-(iii) and (v) hold and suppose there exists a solution $(Y,Z,N)$ to the BSDE \eqref{BSDEF} such that
$Y\in\mathfrak{E}$. Then $(Z,N)\in\mathfrak{M}^p\times\mc{M}^p$ for all $p\geq 1$, more precisely
\begin{gather}
\E\!\left[\!\left(\!\int_0^T\!\!Z_s^\tr\,d\lo
M\ro_sZ_s+d\lo N\ro_s\!\right)^{p/2}\right]\leq
c_{p,\gamma}\,\E\!\left[\exp\!\bigg(\!4p\gamma
e^{\beta^*T}\Big(|\xi|+|\a|_1\!\Big)\!\bigg)\!\right]\!,\label{normbds2}
\end{gather}
where $c_{p,\gamma}$ is a positive constant depending on $p$ and
$\gamma$. The estimate \eqref{normbdsY} then holds as well.
\end{enumerate} 
\end{cor}
\begin{proof}
(i) Let $(Y,Z,N)$ be the solution to \eqref{BSDEF} obtained in
Theorem \ref{ThmExistPr}. As in the previous section set
\begin{equation}\label{deftiH}
\ti{H}_t:=\exp\!\left(\gamma
e^{\beta^*t}|Y_t|+\gamma\int_0^te^{\beta^*r}\,d\lo\lambda\cdot
M\ro_r\right),
\end{equation}
which is a local submartingale. Moreover, from the estimate
\eqref{aposterioriPr}, Jensen's inequality and the adaptedness
of $\int_0^\cdot e^{\beta^*r}\,d\lo\lambda \cdot M\ro_r$ we deduce that
\begin{align*}
\ti{H}_t&=\Big[\exp(\gamma|Y_t|)\Big]^{e^{\,\beta^*t}}\exp\!\left(\gamma\int_0^t e^{\beta^*r}\,d\lo\lambda\cdot M\ro_r \right)\\
&\leq  \E\!\left[\exp\left(\gamma e^{\beta^*(T-t)}|\xi|+\gamma\int_t^Te^{\beta^*(r-t)}\,d\lo\lambda\cdot M\ro_r \right)\!\Bigg|\,\mc{F}_t\right]^{e^{\,\beta^*t}}\exp\!\left(\gamma\int_0^t e^{\beta^*r}\,d\lo\lambda\cdot M\ro_r \right)\\
&\leq\E\!\left[\exp\!\bigg(\gamma e^{\beta^*T}\Big(|\xi|+|\a|_1\!\Big)\!\bigg)\!\Bigg|\,\mc{F}_t\right].
\end{align*}
Observe that this upper estimate is a uniformly integrable martingale, in particular it is of class D and therefore $\ti{H}$ is a true submartingale. Then, via the Doob maximal inequality, we find that for $p>1$
\begin{align}
\E\!\left[e^{p\gamma
Y^*}\right]\leq\E\!\left[\sup_{\substack{0\leq t\leq
T}}\ti{H}_t^p\right]&\leq\left(\frac{p}{p-1}\right)^p\E[\ti{H}_T^p]\notag\\&\leq\left(\frac{p}{p-1}\right)^p\E\!\left[\exp\!\bigg(p\gamma
e^{\beta^*T}\Big(|\xi|+|\a|_1\!\Big)\!\bigg)\right]\label{DoobY},
\end{align}
provided the right hand side is finite. In particular, $e^{\gamma Y}\in\mc{S}^{p^*}$ and $Y\in\mathfrak{E}$ as soon as $|\xi|+\an$ has exponential moments of all orders, in which case \eqref{normbdsY} holds.\\

(ii) We first verify that \eqref{normbdsY} continues to hold when $(Y,Z,N)$ is a solution to \eqref{BSDEF} (assumed to exist) with
$Y\in\mathfrak{E}$. Observe that we may reformulate the result of Proposition
\ref{apriori0} under the condition that
\begin{equation*}
\exp\!\left(\gamma e^{\beta^*T}|Y_\cdot|+\gamma\int_0^{{}^{\,{\boldsymbol{\cdot}}}} e^{\beta^*r}\,d\lo\lambda\cdot M\ro_r \right)
\end{equation*}
be of class D. Repeating the argument from (i) using \eqref{aposteriori} instead of \eqref{aposterioriPr} leads to the same conclusion, since we have the relation
\begin{equation*}
\exp\!\left(\gamma e^{\beta^*t}|Y_t|+\gamma\int_0^{t}e^{\beta^*r}\,d\lo\lambda\cdot M\ro_r \right)\leq\E\!\left[\exp\!\bigg(\gamma e^{\beta^*T}\Big(Y^*+|\a|_1\!\Big)\!\bigg)\!\Bigg|\,\mc{F}_t\right],
\end{equation*}
so that the right hand side is indeed a process of class D. For the
remaining claim, relation \eqref{normbds2}, define the functions $u, v:\R\to\R_+$ via
$u(x):=\tfrac{1}{\gamma^2}(e^{\gamma x}-1-\gamma x)$ and
$v(x):=u(|x|)$. We have that $v$ is a $\mc{C}^2$ function, so we
use It{\^o}'s formula to see that for a stopping
time $\tau$ (to be chosen later)
\begin{multline*}
v(Y_0)=v(Y_{t\wedge\tau})-\int_0^{t\wedge\tau}u'(|Y_s|)\,{\sgn}^*(Y_s)(Z_s^{\tr}\,dM_s+dN_s)\\
+\int_0^{t\wedge\tau}u'(|Y_s|)\,{\sgn}^*(Y_s)\biggl(F(s,Y_s,Z_s)\,dA_s+\frac{1}{2}\,d\lo N\ro_s\biggr)\\-\frac12\int_0^{t\wedge\tau}u''(|Y_s|)\Big( Z_s^{\tr}\,d\lo
M\ro_sZ_s+d\lo N\ro_s  \Big),
\end{multline*}
where use the notation ${\sgn}^*(x):=-\I{\!\{x\leq
0\!\}}+\I{\!\{x>0\!\}}$ and observe that $u'(0)=0$. Assumption \ref{ass} (iii) yields
\begin{multline*}
v(Y_0)\leq
v(Y_{t\wedge\tau})-\int_0^{t\wedge\tau}u'(|Y_s|)\,{\sgn}^*(Y_s)(Z_s^{\tr}\,dM_s+dN_s)\\+\int_0^{t\wedge\tau}u'(|Y_s|)\Big(\a_s+\a_s\beta|Y_s|\Big)dA_s+\frac12
\int_0^{t\wedge\tau}\Big(\gamma u'(|Y_s|)-u''(|Y_s|)
\Big)Z_s^{\tr}\,d\lo
M\ro_sZ_s\\+\frac12\int_0^{t\wedge\tau}\Big(u'(|Y_s|)\,{\sgn}^*(Y_s)-u''(|Y_s|)\Big)\,d\lo
N\ro_s,
\end{multline*}
since $u'(x)=\tfrac{1}{\gamma}(e^{\gamma x}-1)\geq0$ for $x\geq0$.
Using the relation $\gamma u'(x)-u''(x)=-1$ together with
$\gamma\geq 1$ it follows that
\begin{align}
0\leq v(Y_0)&\leq v(Y_{t\wedge\tau})-\int_0^{t\wedge\tau}u'(|Y_s|)\,{\sgn}^*(Y_s)(Z_s^{\tr}\,dM_s+dN_s)\notag\\
&+\int_0^{t\wedge\tau}u'(|Y_s|)\Big(\a_s+\a_s\beta|Y_s|\Big)\,dA_s-\frac12
\int_0^{t\wedge\tau}Z_s^{\tr}\,d\lo M\ro_sZ_s+d\lo
N\ro_s\label{ForBDG}.
\end{align}
Suppose first that $p\geq2$. Then \eqref{normbds2} can be proved
using the Burkholder-Davis-Gundy inequalities as follows. From
\eqref{ForBDG} we deduce that
\begin{multline*}
\frac12 \int_0^{\tau}Z_s^{\tr}\,d\lo M\ro_sZ_s+d\lo
N\ro_s\leq\frac{1}{\gamma^2}\,e^{\gamma
Y^*}+\frac{1}{\gamma}\int_0^{T}e^{\gamma|Y_s|}\Big(\a_s+\a_s\beta|Y_s|\Big)\,dA_s\\+\sup_{\substack{0\leq
t\leq T}}\left|
\int_0^{t\wedge\tau}u'(|Y_s|)\,{\sgn}^*(Y_s)(Z_s^{\tr}\,dM_s+dN_s)\right|,
\end{multline*}
where we used the estimates $u'(x)\leq e^{\gamma x}/\gamma$ and
$v(x)\leq e^{\gamma x}/\gamma^2$, valid for $x\geq0$. From the
inequalities $y\leq e^{y}-1$ and $\beta\leq \gamma$ we derive
\begin{multline*}
\bigg(\int_0^{\tau}Z_s^{\tr}\,d\lo M\ro_sZ_s+d\lo
N\ro_s\bigg)^{p/2}
\leq2^{\,3p/2-2}\Bigg(\frac{1}{\gamma^p}\,e^{p/2\,\gamma
Y^*}+\frac{1}{\gamma^{p/2}}\,e^{p\gamma
Y^*}\an^{p/2}\\+\sup_{\substack{0\leq t\leq T}}\left|
\int_0^{t\wedge\tau}u'(|Y_s|){\sgn}^*(Y_s)(Z_s^{\tr}\,dM_s+dN_s)\right|^{p/2}\Bigg),
\end{multline*}
which yields, after taking expectation and applying the estimate
$|x|^{p/2}<e^{p/2\,|x|}$ and the Burkholder-Davis-Gundy inequality,
\begin{align*}
\E\Bigg[\bigg(\int_0^{\tau}Z_s^{\tr}\,d\lo M\ro_sZ_s+d\lo
N\ro_s\bigg)^{p/2}\Bigg]&\leq
c_{p,\gamma}\,\E\!\left[e^{p/2\,\gamma
Y^*}+e^{p\gamma
Y^*}e^{p/2\,\gamma\an}\right]\\&+c_{p,\gamma}\,\E\!\left[\left(\int_0^{\tau}e^{2\gamma|Y_s|}\Big(Z_s^{\tr}\,d\lo
M\ro_sZ_s+d\lo N\ro_s\Big)\right)^{p/4}\right]\!,
\end{align*}
where we used the estimate $u'(x)\leq e^{\gamma x}/\gamma$ for $x\geq0$. Note that in the above and in what follows $c_{p,\gamma}>0$ is a generic constant depending on $p$ and $\gamma$ that may change
from line to line. We apply the generalized Young inequality,
$|ab|\leq\frac{\eps}{2}\,a^2+\frac{b^2}{2\eps}$, for $\eps:=1$ and for $\eps:=c_{p,\gamma}$. Then, after an adjustment of $c_{p,\gamma}$,
\begin{align*}
\E\Bigg[\bigg(\int_0^{\tau}&Z_s^{\tr}\,d\lo M\ro_sZ_s+d\lo N\ro_s\bigg)^{p/2}\Bigg]\\
&\leq c_{p,\gamma}\left(\E\!\left[e^{p/2\,\gamma
Y^*}\right]+\frac12\,\E\!\left[e^{2p\gamma
Y^*}\right]+\frac12\,\E\!\left[e^{p\gamma\an}\right]\right)\\&
+c_{p,\gamma}\,\E\!\left[e^{p\gamma
Y^*}\right]+\frac12\,\E\Bigg[\bigg(\int_0^{\tau}Z_s^{\tr}\,d\lo
M\ro_sZ_s+d\lo N\ro_s\bigg)^{p/2}\Bigg]\\&\leq
c_{p,\gamma}\left(\E\!\left[e^{2p\gamma
Y^*}\right]+\E\!\left[e^{2p\gamma\an}\right]\right)+\frac12\,\E\Bigg[\bigg(\int_0^{\tau}Z_s^{\tr}\,d\lo
M\ro_sZ_s+d\lo N\ro_s\bigg)^{p/2}\Bigg].
\end{align*}
Next define, for each integer $n\geq 1$, the stopping time
\begin{align*}
\tau_n:=
&\inf\!\left\{t\in[0,T]\,\bigg|\!\int_0^te^{2\gamma|Y_s|}\Big(Z^\tr_s\,d\lo
M\ro_sZ_s+d\lo N\ro_s\Big)\geq n\right\}\wedge T.
\end{align*}
Inserting $\tau_n$ into the above calculation and using
$e^a+e^b\leq2e^{a+b}$ for $a,b\geq0$ together with \eqref{normbdsY},
we may rewrite the last estimate as
\begin{align}
\E\!\left[\bigg(\!\int_0^{\tau_n}\!Z_s^{\tr}\,d\lo
M\ro_sZ_s+d\lo N\ro_s\!\bigg)^{p/2}\right] \leq
c_{p,\gamma}\,\E\!\left[\exp\!\bigg(\!2p\gamma
e^{\beta^*T}\Big(|\xi|+|\a|_1\!\Big)\!\bigg)\!\right]\!.\label{normbds2-}
\end{align}
By Fatou's lemma, since $\tau_n\to T$ as $n\to+\infty$,
\begin{align*}
\E\!\left[\bigg(\!\int_0^T\!Z_s^{\tr}\,d\lo M\ro_sZ_s+d\lo
N\ro_s\!\bigg)^{p/2}\right] \leq
c_{p,\gamma}\,\E\!\left[\exp\!\bigg(\!2p\gamma
e^{\beta^*T}\Big(|\xi|+|\a|_1\!\Big)\!\bigg)\!\right]\!
\end{align*}
and \eqref{normbds2} follows. In the situation where $p<2$,
$q:=2/p>1$ and we may combine Jensen's inequality with
\eqref{normbds2-}, which is valid for $p=2$, to get
\begin{multline*}
\E\Bigg[\bigg(\int_0^{\tau_n}Z_s^{\tr}\,d\lo M\ro_sZ_s+d\lo
N\ro_s\bigg)^{p/2}\Bigg]^q\\\leq
\E\Bigg[\!\int_0^{\tau_n}\!Z_s^{\tr}\,d\lo M\ro_sZ_s+d\lo
N\ro_s\!\Bigg]\leq
c_{2,\gamma}\,\E\!\left[\exp\!\!\bigg(\!4\gamma
e^{\beta^*T}\Big(|\xi|+|\a|_1\Big)\!\bigg)\!\right]
\end{multline*}
from which \eqref{normbds2} follows after another application of
Fatou's lemma together with the fact that the right hand side in the
inequality above is greater or equal one while $1/q=p/2<1$.
\end{proof}
\begin{rmk}
We point out that the results of this section do not require that
$F$ be convex in $z$, but only that $F$ be continuous in $(y,z)$. The reader may have noticed that the continuity of $F$ is not used directly in the proofs. However in Theorem \ref{ThmExistPr} we rely on the results of \cite{Mo09} where continuity is a technical condition needed for an application of Dini's theorem. In addition our results also apply to the BSDE \eqref{BSDEF-} if $g$ is identically
equal to a nonzero constant $\gamma_g/2$, in which case we assume without loss of generality
that $\gamma\geq |\gamma_g|$.
\end{rmk}


\section{Uniqueness}\label{Uniq}
We now provide a comparison theorem that yields uniqueness of the
BSDE triple. The proof makes use of the $\theta$-technique applied in the context  
of second order Bellman-Isaacs equations by Da Lio and Ley \cite{DL06} and subsequently
adapted to the framework of Brownian BSDEs in \cite{BH08}.

\begin{thm}[Comparison Principle]\label{comp} Let $(Y,Z,N)$ and $(Y',Z',N')$ be solutions to the BSDE \eqref{BSDEF} with drivers $F$ and $F'$ and terminal conditions $\xi$ and $\xi'$, respectively. Suppose in addition that $Y\in\mathfrak{E}$ and $Y'\in\mathfrak{E}$.
If $\mb{P}$-a.s.
\[\xi\leq\xi'\quad\text{and}\quad F(t,y,z)\leq F'(t,y,z) \text{ for all }(t,y,z)\in[0,T]\times\R\times\R^d\]
and if $(F,\xi)$ satisfies Assumption \ref{ass} (i)-(iii) then
$\mb{P}$-a.s. for each $t\in[0,T]$
\[Y_t\leq Y'_t.\]
\end{thm}

\begin{proof}
Let $\theta$ be a real number in $(0,1)$ and set $U:=Y-\theta Y'$,
$V:=Z-\theta Z'$ and $W:=N-\theta N'$. Consider the process
\[ \rho_s:=
\begin{cases}\, \frac{F(s,Y_s,Z_s)-F(s,\theta Y_s',Z_s)}{U_s}& \text{if }U_s \neq 0, \\
    \,  \ol{\beta}& \text{if }U_s= 0.
\end{cases}\]
By Assumption \ref{ass} (ii), $\rho$ is bounded by $\ol{\beta}$ and
we define $R_s:=\int_0^s\rho_r\,dA_r$. Notice that by the boundedness of $A$ we have that $|R|\leq\ol{\beta}\,A_T\leq
\ol{\beta}K_A$. From It{\^o}'s formula we deduce
\begin{multline*}
e^{R_t}U_t=e^{R_T}U_T-\int_t^Te^{R_s}(V_s^\tr\,dM_s+dW_s)+\int_t^Te^{R_s}\bigg(F^\theta_s\,dA_s+\frac12\Big(d\lo
N\ro_s-\theta\,d\lo N'\ro_s\Big)\bigg),
\end{multline*}
where we define $F^\theta_s:=F(s,Y_s,Z_s)-\theta
F'(s,Y'_s,Z_s')-\rho_sU_s$. We also set \[\Delta
F(s):=(F-F')(s,Y_s',Z_s')\leq 0\] and observe that from the
convexity of $F$ in $z$ together with \eqref{Ftyzbetas} we get
\begin{align}
F(s,Y'_s,&Z_s)-\theta F(s,Y_s',Z'_s)= F\!\left(s,Y_s',\theta
Z_s'+(1-\theta)\frac{Z_s-\theta Z_s'}{1-\theta}\right)-\theta
F(s,Y_s',Z'_s)\notag\\&\leq (1-\theta)\,F\!\left(s,Y_s',\frac{Z_s-\theta
Z_s'}{1-\theta}\right) \leq(1-\theta)\a_s+
(1-\theta)\ol{\beta}|Y_s'|+\frac{\gamma}{2(1-\theta)}\,\|B_sV_s\|^2.\label{FF1}
\end{align}
Another application of the Lipschitz assumption \ref{ass} (ii),
yields
\begin{align}
F(s,Y_s,Z_s)-F(s,Y_s',Z_s)&= F(s,Y_s,Z_s)-F(s,\theta
Y_s',Z_s)+F(s,\theta Y'_s,Z_s)-F(s, Y_s',Z_s)\notag\\&=
\rho_sU_s+F(s,\theta Y'_s,Z_s)-F(s, Y_s',Z_s)\notag\\&\leq
\rho_sU_s+(1-\theta)\ol{\beta}|Y_s'|.\label{FF2}
\end{align}
Combining \eqref{FF1} and \eqref{FF2} we see that
\begin{align}
F^\theta_s&=F(s,Y_s,Z_s)-\theta F(s,Y_s',Z'_s)+\theta\, \Delta F(s)-\rho_sU_s\notag\\
&=[F(s,Y_s,Z_s)-F(s,Y_s',Z_s)]+[F(s,Y'_s,Z_s)-\theta F(s,Y_s',Z'_s)]+\theta\, \Delta F(s)-\rho_sU_s\notag\\
&\leq
(1-\theta)\Big(\a_s+2\ol{\beta}|Y_s'|\Big)+\frac{\gamma}{2(1-\theta)}\,\|B_sV_s\|^2+\theta
\,\Delta F(s).\label{Fthbdd}
\end{align}
Let $\kappa:=\frac{\gamma
\exp\left(\ol{\beta}K_A\right)}{1-\theta}>0$ and
$P_t:=\exp\!\left(\kappa e^{R_t}U_t\right)>0$. The logic is now
similar to how we derived the a priori estimates, namely to show
that, by removing an appropriate drift, $P$ is a (local)
submartingale. By It{\^o}'s formula, for $t\in[0,T]$,
\begin{align}
P_t&=P_T-\int_t^T\kappa P_s
e^{R_s}(V_s^\tr\,dM_s+dW_s)+\int_t^T\kappa P_s
e^{R_s}\left(F^\theta_s-\frac{\kappa
e^{R_s}}{2}\,\|B_sV_s\|^2\right)\,dA_s\label{Ftheta}\\&+\int_t^T\kappa
P_s e^{R_s}\bigg(-\frac{\kappa e^{R_s}}{2}\,d\lo
W\ro_s+\frac12\Big(d\lo N\ro_s-\theta\,d\lo
N'\ro_s\Big)\bigg)\label{gsN}.
\end{align}
To simplify notation set
\begin{gather}
G:=\kappa P e^{R}\left(F^\theta-\frac{\kappa
e^{R}}{2}\,\|BV\|^2\right)\label{GdA} \quad\text{and}\\
H:=\int_0^{{}^{\,\boldsymbol{\cdot}}}\kappa P_s e^{R_s}\bigg(-\frac{\kappa
e^{R_s}}{2}\,d\lo W\ro_s+\frac12\Big(d\lo N\ro_s-\theta\,d\lo
N'\ro_s\Big)\bigg).\label{HdA}
\end{gather}
Let us first investigate the finite variation process $H$. We
claim that $H$ is decreasing, indeed for all $r, u\in[0,T]$, $r\leq u$,
we have
\[\int_r^ud\lo W\ro_s=\int_r^ud\lo N\ro_s-2\theta\,d\lo N,N'\ro_s+\theta^2\,d\lo N'\ro_s.\]
Applying the Kunita-Watanabe and Young inequalities,
\begin{align*}
\int_r^ud\lo W\ro_s&\geq\int_r^ud\lo
N\ro_s+\int_r^u\theta^2\,d\lo
N'\ro_s-2\theta\left(\int_r^ud\lo
N\ro_s\right)^{1/2}\left(\int_r^ud\lo
N'\ro_s\right)^{1/2}\\&\geq\int_r^ud\lo
N\ro_s+\int_r^u\theta^2\,d\lo
N'\ro_s-\theta\left(\int_r^ud\lo N\ro_s+\int_r^ud\lo
N'\ro_s\right)\\&=(1-\theta)\left(\int_r^ud\lo
N\ro_s-\theta\,d\lo N'\ro_s\right).
\end{align*}
In particular, since $\gamma\geq1$ and $|R|\leq\ol{\beta}K_A$ we have,
\[\int_r^u\kappa e^{R_s}d\lo W\ro_s\geq\frac{\gamma}{1-\theta}\int_r^ud\lo W\ro_s\geq\int_r^ud\lo N\ro_s-\theta\,d\lo N'\ro_s,\]
which shows that the process $H$ is decreasing and hence the integral in
\eqref{gsN} is nonpositive.

Next we consider the finite variation integral in \eqref{Ftheta}.
Combining \eqref{Fthbdd}, $\Delta F\leq 0$ and the boundedness of $R$ we have
\begin{equation}
G=\kappa P e^{R}\left(F^\theta-\frac{\kappa
e^{R}}{2}\,\|BV\|^2\right)
\leq\kappa P
e^{R}\left((1-\theta)\Big(\a+2\ol{\beta}|Y'|\Big)\right)\leq
PJ\label{GPJ},
\end{equation}
where \[ J:=\gamma
e^{2\ol{\beta}K_A}\Big(\a+2\ol{\beta}|Y'|\Big)\geq0.\] We set
\[D_t:=\exp\left(\int_0^tJ_s\,dA_s\right)\quad\text{ and
}\quad\ti{P}_t:=D_tP_t.\] Partial integration yields
\begin{align}
d\ti{P}_t&=D_t\Big(\!-G_t\,dA_t-dH_t+\kappa
P_te^{R_t}\!\left(V_t^{\tr}\,dM_t+dW_t\right)\!\Big)+P_t D_t
J_t\,dA_t\notag\\&=D_t\Big((P_tJ_t-G_t)\,dA_t-dH_t+\kappa
P_te^{R_t}\!\left(V_t^{\tr}\,dM_t+dW_t\right)\!\Big)\label{Partial}
\end{align}
and we conclude that the finite variation parts in the latter
expression are nonnegative. We can now use the following stopping
time argument to derive
\begin{equation}
P_t\leq\E\!\left[\frac{D_T}{D_t}\,P_T\bigg|\,\mc{F}_t\right].
\label{PtDtDT}
\end{equation}
Namely, consider the stopping time
\begin{align*}
\tau_n:=
\inf\!\left\{u\in[t,T]\,\bigg|\!\int_t^u\!\kappa^2\ti{P}^2_s\,e^{2R_s}\Big(V^\tr_s\,d\lo
M\ro_sV_s+d\lo W\ro_s\Big)\geq n\right\}\wedge T,
\end{align*}
where $n\geq 1$ is an integer. Observe that $\tau_n\to T$ as $n\to+\infty$ due to the integrability assumptions on $\a$, $Y$ and $Y'$, as well as the boundedness of $A$. Then \eqref{Partial} provides the
estimate
\begin{align*}
P_t\leq\E\!\left[\exp\!\left(\int_t^{\tau_n}\!J_s\,dA_s\right)P_{\tau_n}\bigg|\,\mc{F}_t\right]
=\E\!\left[\exp\!\left(\int_t^{\tau_n}\!\gamma
e^{2\ol{\beta}K_A}\Big(\a_s+2\ol{\beta}|Y'_s|\Big)\,dA_s\right)\!P_{\tau_n}\bigg|\,\mc{F}_t\right].
\end{align*}
In view of the current integrability and boundedness assumptions we can send $n$ to infinity and deduce
\eqref{PtDtDT}.

Notice that we also have
$\xi-\theta\xi'\leq(1-\theta)|\xi|+\theta\Delta\xi$, where
$\Delta\xi:=\xi-\xi'\leq0$. Then together with the definition of $P$ \eqref{PtDtDT} shows that
\begin{align*}
\exp\!\Bigg(\frac{\gamma
e^{\ol{\beta}K_A+R_t}}{1-\theta}\,&\left(Y_t-\theta
Y'_t\right)\Bigg)\\&
\leq\E\Bigg[\exp\!\left(\int_t^T\!\gamma
e^{2\ol{\beta}K_A}\Big(\a_s+2\ol{\beta}|Y'_s|\Big)\,dA_s\right)
\exp\!\left(\kappa
e^{R_T}\!\left(\xi-\theta\xi'\right)\right)\!\Bigg|\,\mc{F}_t\Bigg]\\&\leq\E\Bigg[\exp\!\left(\gamma
e^{2\ol{\beta}K_A}\int_t^T
\Big(\a_s+2\ol{\beta}|Y'_s|\Big)\,dA_s\right)\exp\!\left( \gamma
e^{2\ol{\beta}K_A}|\xi|\right)\!\Bigg|\,\mc{F}_t\Bigg].
\end{align*}
Thus, we can derive the estimate
\begin{equation*}
Y_t-\theta
Y'_t\leq\frac{1-\theta}{\gamma}\log\E\Bigg[\exp\!\left(\gamma
e^{2\ol{\beta}K_A}\left(|\xi|+\int_t^T
\Big(\a_s+2\ol{\beta}|Y'_s|\Big)\,dA_s\right)\right)\!\Bigg|\,\mc{F}_t\Bigg],
\end{equation*}
which follows from the above by checking the cases $Y_t-\theta
Y'_t\geq0$ and $Y_t-\theta
Y'_t<0$ separately, noting that $R+\ol{\beta}K_A\geq0$.
Once again, by the integrability assumptions on $\xi$, $\a$ and
$Y'$ and the boundedness of $A$, the conditional expectation on
the right hand side is finite, $\mb{P}$-a.s. Taking $\theta
\uparrow 1$ then gives $Y_t\leq Y'_t$ and the continuity of $Y$ and $Y'$ yields the claim. 
\end{proof}

The following corollary is then immediate.

\begin{cor}[Uniqueness]\label{ExistUniq}
\hspace{-2.5mm}Let Assumption \ref{ass} (i)-(iii) hold and
let $(Y,Z,N)$ and $(Y',Z',N')$ be two solutions to the BSDE
\eqref{BSDEF} with $Y\in\mathfrak{E}$ and
$Y'\in\mathfrak{E}$. Then $Y$ and $Y'$, $Z\cdot M$ and $Z'\cdot M$, and $N$ and $N'$ are indistinguishable. In addition $(Z\cdot M,N)$ and  $(Z'\cdot M,N')$ both belong to $\mc{M}^p\times\mc{M}^p$ for all $p\geq1$.
\end{cor}

\begin{proof}
By Theorem \ref{comp} and Corollary \ref{Cornormbds} (ii) only the
assertion regarding the indistinguishability of the martingale part remains. It{\^o}'s formula gives $\mb{P}$-a.s.
\begin{multline*}
0=(Y_T-Y_T')^2=(Y_0-Y_0')^2+2\int_0^T(Y_t-Y_t')\,d(Y_t-Y_t')\\+\int_0^T(Z_t-Z'_t)^\tr\,d\lo M\ro_t(Z_t-Z'_t)+d\lo N-N'\ro_t\\
=\int_0^T(Z_t-Z'_t)^\tr\,d\lo M\ro_t(Z_t-Z'_t)+d\lo N-N'\ro_t,
\end{multline*}
from which $Z\cdot M\equiv Z'\cdot M$ and $N\equiv N'$.
\end{proof}



\section{Stability}\label{Stab}
It follows from the previous results that the BSDE \eqref{BSDEF}
has a unique solution in $\mathfrak{E}$ under appropriate
Lipschitz and convexity assumptions on the driver $F$ and an
exponential moments condition on the terminal value $\xi$ and
process $\a$. In the present section we show that a stability
result for such BSDEs also holds. More precisely, given a
sequence of terminal values and a sequence of drivers such that
the exponential moments condition is fulfilled uniformly, and such that they both converge to a fixed terminal value and a
fixed generator in a suitable sense, then we gain convergence on
the level of the respective BSDE solutions. 
\begin{thm}[Stability]\label{thmstab}
Let $(F^n)_{n\geq0}$ be a sequence of generators for the BSDE
\eqref{BSDEF} such that Assumption \ref{ass} (ii)-(iii) and (v) hold for each $F^n$ with the set of parameters
$(\a^n,\beta^n,\ol{\beta},\beta_f,\gamma)$. If $(\xi^n)_{n\geq 0}$
are the associated random terminal values then suppose that, for
each $p>0$,
\begin{equation}
\sup_{\substack{n\geq 0}}\E\!\left[\,e^{p\,(|\xi^n|+|\a^n|_1)}\right]<+\infty.\label{supxin}
\end{equation}
Let $(Y^n,Z^n,N^n)$ be the solution to the BSDE \eqref{BSDEF} with
driver $F^n$ and terminal condition $\xi^n$ such that
$Y^n\in\mathfrak{E}$ for all $n\geq0$. If
\begin{equation}
|\xi^n-\xi^0| + \int_0^T\big|F^n-F^0\big|\,(s,Y_s^0,Z_s^0)\,dA_s\longrightarrow 0 \quad\text{ in
probability, as }n\to+\infty,\label{xiFninprob}
\end{equation}
then for each $p>0$,
\begin{gather*}
\lim_{n\to+\infty}\E\Bigg[\!\!\left(\exp\!\left(\sup_{\substack{0\leq t\leq T}}\big|Y_t^n-Y_t^0\big|\right)\!\right)^p\Bigg]= 1\quad\text{and}\\
\lim_{n\to+\infty}\E\Bigg[\!\!\left(\int_0^T(Z_s^n-Z_s^0)^{\tr}\,d\lo
M\ro_s(Z_s^n-Z_s^0)+d\lo
N^n-N^0\ro_s\right)^{p/2}\Bigg]= 0.
\end{gather*}
\end{thm}
\begin{rmk}
Let us briefly indicate how our present stability theorem differs from those in the literature, \cite{Fr09} Theorem 2.1 and \cite{Mo09} Lemma 3.3. The key points are that firstly in our conditions the parameters $\alpha^n$ and $\beta^n$ are allowed to depend on $n$, whereas in \cite{Fr09} and \cite{Mo09} they are assumed independent of $n$. Secondly, we assume a uniform exponential moments condition, as opposed to a uniform boundedness condition in the cited references. 
Finally, in the unbounded setting we require the mode of convergence assumed above, this is in contrast to the bounded setting of \cite{Fr09} Theorem 2.1 where the weaker notion of pointwise convergence is sufficient for a stability result to hold.
\end{rmk}
\begin{proof} Note that Assumption \ref{ass} (i) holds for each $n$ thanks to \eqref{supxin}. Exactly as in the statement of Corollary \ref{Cornormbds} we
deduce that for each $p\geq1$
\begin{align*}
\sup_{\substack{n\geq
0}}\E\Bigg[\!\!\left(\exp\!\left(\sup_{\substack{0\leq t\leq
T}}\big|Y_t^n\big|\right)\!\right)^p+\left(\int_0^T(Z_s^n)^\tr\,d\lo
M\ro_sZ_s^n+d\lo N^n\ro_s\right)^{p/2}\Bigg]<+\infty.
\end{align*}
Hence the sequences in $n$ of random variables
\begin{align*}
\left(\exp\!\left(\sup_{\substack{0\leq t\leq
T}}\big|Y_t^n\big|\right)\!\right)^p\quad\text{ and }\quad\left(\int_0^T(Z_s^n)^\tr\,d\lo
M\ro_sZ_s^n+d\lo N^n\ro_s\right)^{p/2}
\end{align*}
are uniformly integrable for all $p\geq1$. 
By the Vitali convergence theorem, it is thus sufficient to prove that
\begin{equation*}
\sup_{\substack{0\leq t\leq
T}}\big|Y_t^n-Y^0_t\big|+\int_0^T(Z_s^n-Z^0_s)^\tr\,d\lo
M\ro_s(Z_s^n-Z^0_s)+d\lo
N^n-N^0\ro_s\longrightarrow 0
\end{equation*}
in probability when $n$ tends to infinity.\\

We split the proof into four steps. The first two steps construct one-sided estimates for the difference of $Y^n$ and $Y^0$ proceeding very similarly to the proof of the comparison result. In the third step we combine the aforementioned estimates to show that $Y^n-Y^0$ converges to zero uniformly on $[0,T]$ in probability, i.e. in \emph{ucp}. Finally, we use this result to show the required convergence of the martingale parts.\\

\emph{Step 1.} First fix $\theta\in(0,1)$ and $n\geq1$ and
proceed in the same way as in the proof of Theorem \ref{comp} by
defining the same objects $U$, $V$, $W$, $\rho$, $R$, $F^\theta$,
$\kappa$, $P$, $G$, and $H$, subject to the following
modification. All the objects $X'\in\{Y',Z',N',F'\}$ with a prime
${}'$ are replaced by the respective object $X^0$ with a
superscript 0. All the objects $X\in\{Y,Z,N,F,\a\}$ without a
prime are replaced by the respective object $X^n$ with a
superscript $n$, e.g. $U:=Y^n-\theta Y^0$. We observe that the above objects depend on $n$ however we omit this dependence for brevity. In addition set
$\Delta^n F(s)=(F^n-F^0)(s,Y^0_s,Z^0_s)$. From \eqref{Fthbdd}
and \eqref{GdA},
\begin{align*}
G&\leq\kappa P
e^{R}\left((1-\theta)\Big(\a^n+2\ol{\beta}|Y^0|\Big)+\theta\,
\Delta^n F\right)\notag\\&\leq\gamma e^{2\ol{\beta}K_A} P
\left(\a^n+2\ol{\beta}|Y^0|+\frac{|\Delta^n
F|}{1-\theta}\right)=PJ^n+\gamma e^{2\ol{\beta}K_A} P
\frac{|\Delta^n F|}{1-\theta},
\end{align*}
where, consistent with our modification, \[ J^n:=\gamma
e^{2\ol{\beta}K_A} \left(\a^n+2\ol{\beta}|Y^0|\right)\geq 0.\]
Considering
\[D^n_t:=\exp\left(\int_0^tJ^n_s\,dA_s\right)\quad\text{ and
}\quad\ti{P}^n_t:=D^n_tP_t\] and applying partial integration
yields
\begin{multline*}
d\ti{P}^n_t+\gamma e^{2\ol{\beta}K_A} \ti{P}^n_t \frac{|\Delta^n
F|}{1-\theta}\,dA_t =D^n_t\,dP_t+P_t\,dD^n_t+\gamma
e^{2\ol{\beta}K_A} \ti{P}^n_t \frac{|\Delta^n
F|}{1-\theta}\,dA_t
\\
=D^n_t\!\left[\!\left(P_tJ^n_t+\gamma
e^{2\ol{\beta}K_A} P_t \frac{|\Delta^n
F|}{1-\theta}-G_t\right)dA_t-dH_t+\kappa
P_te^{R_t}\!\left(V_t^{\tr}\,dM_t+dW_t\right)\right].
\end{multline*}
We conclude that the finite variation parts in the last
expression are nonnegative. We now use the stopping time argument
from the proof of Theorem \ref{comp} to derive
\begin{equation}
P_t\leq D^n_tP_t\leq\E\!\left[D^n_TP_T+ \frac{\gamma
e^{2\ol{\beta}K_A}}{1-\theta}\int_t^T D_s^nP_s |\Delta^n F(s)|\,
dA_s\Bigg|\,\mc{F}_t\right].\label{PtDtDT1}
\end{equation}
From the boundedness of $\rho$ and the definition
$P=\exp\!\left(\kappa e^RU\right)$ we derive, for $s\in[0,T]$,
\begin{gather*}
P_s\leq \sup_{\substack{0\leq t\leq T}}\left[\exp\!\left( \frac{\gamma e^{2\ol{\beta}K_A}}{1-\theta}\Big(|Y_t^0|+|Y_t^n|\Big)\!\right)\right]=:\Upsilon^n(\theta)\quad\text{and}\\
P_T\leq\exp\!\left(\frac{\gamma
e^{2\ol{\beta}K_A}}{1-\theta}\big|\xi^n-\theta
\xi^0\big|\!\right)\leq\exp\!\left(\frac{\gamma
e^{2\ol{\beta}K_A}}{1-\theta}\Big(\big|\xi^n-\theta
\xi^0\big|\vee\big|\xi^0-\theta
\xi^n\big|\Big)\!\right)=:\chi^n(\theta).
\end{gather*}
Using the boundedness of $\rho$, the inequalities $\log(x)\leq x$,
\eqref{PtDtDT1} and $1\leq D_s^n\leq D_T^n$ we then find
\begin{align}
Y_t^n-Y^0_t\notag&\leq(1-\theta)|Y^0_t|+ Y_t^n-\theta
Y_t^0=(1-\theta)|Y_t^0|+U_t\notag\\&=(1-\theta)|Y_t^0|+\frac{1-\theta}{\gamma}\,\exp\!\left(-\ol{\beta}K_A-R_t\right)\log(P_t)\notag\\
&\leq(1-\theta)|Y_t^0|+\frac{1-\theta}{\gamma}\,\E\!\left[D^n_T\chi^n(\theta)+
\frac{\gamma
e^{2\ol{\beta}K_A}}{1-\theta}\,D_T^n\Upsilon^n(\theta)\int_t^T
|\Delta^n F(s)|\, dA_s\Bigg|\,\mc{F}_t\right]\!. \label{StabYbound1}
\end{align}

\emph{Step 2.} With regards to the converse inequality we proceed
as in the proof of Theorem \ref{comp} again. We define the same
objects $U$, $V$, $W$, $R$, $F^\theta$, $\kappa$, $P$, $G$, and
$H$ but now subject to the following modification. All the objects
$X'\in\{Y',Z',N',F'\}$ with a prime ${}'$ are replaced by the
respective object $X^n$ with a superscript $n\geq1$. All the objects
$X\in\{Y,Z,N,F,\a\}$ without a prime are replaced by the
respective object $X^0$ with a superscript 0, e.g. $U:=Y^0-\theta
Y^n$. Moreover, we define $\rho$ differently, namely,
\[\rho_s:=\frac{F^n(s,Y^0_s,Z_s^n)-F^n(s,Y_s^n,Z_s^n)}{Y^0_s-Y_s^n}\,\I{\{|Y^0_s-Y^n_s|>0 \}}.\]
This ensures that $|\rho|\leq \ol{\beta}$ still holds and
\[ \theta F^n(s,Y^0_s,Z_s^n)-\theta F^n(s,Y_s^n,Z_s^n)=\rho(\theta Y^0-\theta Y^n)\leq \ol{\beta}(1-\theta)|Y^0_s|+\rho_s U_s.\]
Using the convexity of $F^n$, the estimate \eqref{Ftyzbetas} and
the \emph{same} definition of $\Delta^n F$ as in Step 1 we derive
\[F^\theta\leq |\Delta^n F|+(1-\theta)(\a^n+2\ol{\beta}|Y^0|)+\frac{\gamma}{2(1-\theta)}\,\|BV\|^2,\]
so that the following inequality holds
\[G\leq\gamma e^{2\ol{\beta}K_A}P \left(\a^n+2\ol{\beta}|Y^0|+\frac{|\Delta^n F|}{1-\theta}\right).\]
Observe that this is the same estimate on $G$ as that obtained in
Step 1. Thus we may rewrite \eqref{StabYbound1} as
\begin{align}
Y^0_t-Y^n_t\leq(1-\theta)|Y_t^n|+\frac{1-\theta}{\gamma}\,\E\!\left[D^n_T\chi^n(\theta)+
\frac{\gamma
e^{2\ol{\beta}K_A}}{1-\theta}\,D_T^n\Upsilon^n(\theta)\int_t^T
|\Delta^n F(s)|\, dA_s\Bigg|\,\mc{F}_t\right]\!,
\label{StabYbound2}
\end{align}
where $J^n$ and thus $D^n$, $\Upsilon^n$ and $\chi^n(\theta)$ are as in Step 1, as well.

\emph{Step 3.} Let us now prove that $\left(\sup_{\substack{0\leq
t\leq T}}\big |Y_t^n-Y_t^0\big|\right)_{n\geq1}$ converges to zero
in probability. Summing up \eqref{StabYbound1} and
\eqref{StabYbound2} we deduce
\begin{multline*}
\big|Y_t^n-Y^0_t\big|\leq(1-\theta)\left(|Y^0_t|+|Y_t^n|\right)+\frac{1-\theta}{\gamma}\,\E\!\left[D^n_T\chi^n(\theta)\bigg|\,\mc{F}_t\right]\notag\\+e^{2\ol{\beta}K_A}\,\E\!\left[
D_T^n\Upsilon^n(\theta)\int_t^T |\Delta^n F(s)|\,
dA_s\Bigg|\,\mc{F}_t\right].
\end{multline*}
Applying the Doob, Markov and H{\"o}lder inequalities, we deduce
the existence of some nonnegative constants $c_1$, $c_2$, independent of $\theta$, as well as
$c_3(\theta)$ such that, for $\eps>0$,
\begin{align}
\mb{P}\Bigg(&\sup_{\substack{0\leq t\leq T}}\big
|Y_t^n-Y_t^0\big|\geq \eps\Bigg)\notag\\&
\leq\mb{P}\!\left((1-\theta)\sup_{\substack{0\leq t\leq T}}
\left(|Y_t^0|+|Y_t^n|\right)\geq
\frac{\eps}{3}\right)+\mb{P}\!\left(\frac{1-\theta}{\gamma}\sup_{\substack{0\leq t\leq T}}\E\!\left[D^n_T\chi^n(\theta)\bigg|\,\mc{F}_t\right]\geq\frac{\eps}{3}\right)\notag\\&\hspace{4cm}+\mb{P}\!\left(e^{2\ol{\beta}K_A}\sup_{\substack{0\leq t\leq T}}\E\!\left[
D_T^n\Upsilon^n(\theta)\int_t^T |\Delta^n F(s)|\,
dA_s\Bigg|\,\mc{F}_t\right]\geq \frac{\eps}{3}\right)\notag\\&
\leq\frac{3(1-\theta)}{\eps}\,\E\!\left[\sup_{\substack{0\leq
t\leq T}}
\left(|Y_t^0|+|Y_t^n|\right)\right]+\frac{3(1-\theta)}{\eps\gamma}\,\E\!\left[D^n_T\chi^n(\theta)\right]\notag\\&\hspace{4cm}+\frac{3e^{2\ol{\beta}K_A}}{\eps}\,\E\!\left[
D_T^n\Upsilon^n(\theta)\int_0^T |\Delta^n F(s)|\,
dA_s\right]\notag\\&\leq
\frac{c_1(1-\theta)}{\eps}+\frac{c_2(1-\theta)}{\eps}\,\E\!\left[\chi^n(\theta)^2\right]^{1/2}+\frac{c_3(\theta)}{\eps}\,\E\!\left[\!\left(\int_0^T
|\Delta^n F(s)|\,dA_s\right)^2\right]^{1/2}\!\!,\label{C1C2C3}
\end{align}
where the latter inequality is due to the fact that, by our
assumptions, the sequences $\left(\sup_{\substack{0\leq t\leq T}}
\left(|Y_t^0|+|Y_t^n|\right)\right)_{n\geq 1}$, $(D_T^n)_{n\geq
1}$ and $(\Upsilon^n(\theta))_{n\geq1}$ are bounded in all
$L^p(\mb{P})$ spaces, $p\geq1$.
In addition, we used that
\begin{align}
|\Delta^n
F(s)|&\leq|F^n(s,Y^0_s,Z^0_s)-F^n(s,0,Z^0_s)|+|F^n(s,0,Z^0_s)|+|F^0(s,0,Z^0_s)|\notag\\&+|F^0(s,Y^0_s,Z^0_s)-F^0(s,0,Z^0_s)|\leq2\ol{\beta}|Y^0_s|+\a^n_s+\a_s^0+\gamma\|B_sZ_s^0\|^2,\label{DeltanFs}
\end{align}
which implies that for all $p\geq1$ $\left(\int_0^T|\Delta^n F(s)|\,dA_s\right)^p$ is uniformly integrable due to \eqref{normbds2} and \eqref{supxin}. Hence, the Vitali convergence theorem and \eqref{xiFninprob} imply that $\int_0^T|\Delta^n F(s)|\,dA_s\to0$ in all $L^p(\mb{P})$ spaces. Observe that $\chi^n(\theta)$ converges in
probability to $\exp\!\left(\gamma
e^{2\ol{\beta}K_A}|\xi^0|\right)$ as $n$ goes to infinity. This
convergence is also in all $L^p(\mb{P})$ spaces because of the
Vitali convergence theorem and our uniform integrability
assumption on $(\xi^n)_{n\geq1}$. More precisely, for all $p\geq1$, we have
\begin{align*}
\sup_{\substack{n\geq1}}\E\!\left[\chi^n(\theta)^p\right]&\leq
\sup_{\substack{n\geq1}}\E\!\left[\exp\!\left(\frac{p\gamma
e^{2\ol{\beta}K_A}}{1-\theta}\Big(\big|\xi^n\big|+|\xi^0|\Big)\!\right)\right]\\&\leq
\sup_{\substack{n\geq1}}\E\!\left[\exp\!\left(\frac{2p\gamma
e^{2\ol{\beta}K_A}}{1-\theta}\,\big|\xi^n\big|\!\right)\right]^{1/2}\E\!\left[\exp\!\left(\frac{2p\gamma
e^{2\ol{\beta}K_A}}{1-\theta}\,\big|\xi^0|\!\right)\right]^{1/2}<+\infty.
\end{align*}
From \eqref{C1C2C3} we then deduce that for all $\theta\in(0,1)$
\begin{equation*}
\limsup_{\substack{n\to+\infty}}\mb{P}\!\left(\sup_{\substack{0\leq
t\leq T}}\big |Y_t^n-Y^0_t\big|\geq
\eps\right)\leq\frac{c_1(1-\theta)}{\eps}+\frac{c_2(1-\theta)}{\eps}\,\E\!\left[\exp\!\left(2\gamma
e^{2\ol{\beta}K_A}|\xi^0|\right)\right]^{1/2}\!.
\end{equation*}
We then send $\theta$ to 1 to conclude that
\begin{equation*}
\lim_{\substack{n\to+\infty}}\mb{P}\!\left(\sup_{\substack{0\leq
t\leq T}}\big |Y_t^n-Y^0_t\big|\geq \eps\right)= 0.
\end{equation*}

\emph{Step 4.} Let us now turn to the last assertion of the
theorem. We derive from It{\^o}'s formula that
\begin{multline*}
\E\Bigg[\!\int_0^T(Z_s^n-Z^0_s)^{\tr}\,d\lo
M\ro_s(Z_s^n-Z^0_s)+d\lo
N^n-N^0\ro_s\Bigg]\\\leq\E\!\left[(\xi^n-\xi^0)^2+2\sup_{\substack{0\leq
t\leq T}}\big
|Y_t^n-Y_t^0\big|\cdot\int_0^T\big|F^n(s,Y_s^n,Z_s^n)-F^0(s,Y^0_s,Z^0_s)\big|\,
dA_s\right]\\+\E\!\left[\sup_{\substack{0\leq t\leq T}}\big
|Y_t^n-Y^0_t\big|\cdot\Bigg|\int_0^Td\lo N^n\ro_s-d\lo
N^0\ro_s\Bigg|\right]
\end{multline*}
after observing that the local martingale arising therein is in fact a true martingale thanks to the present integrability assumptions, cf. Corollary \ref{Cornormbds}. By \eqref{Ftyzbetas},
\begin{multline*}
\big|F^n(s,Y_s^n,Z_s^n)-F^0(s,Y^0_s,Z^0_s)\big|\leq
\a_s^n+\a^0_s+\ol{\beta}|Y_s^n|+\ol{\beta}|Y^0_s|+\frac{\gamma}{2}\,\|B_sZ_s^n\|^2+\frac{\gamma}{2}\,\|B_sZ^0_s\|^2.
\end{multline*}
Applying H{\"o}lder's inequality, the formula
\eqref{normbds2} and the condition \eqref{supxin} we recognize
(the expectation of the squares of) the integrals from the right
hand side above as uniformly bounded (in $n$). The result now
follows from the fact that $\xi^n\to\xi^0$ and that by
Steps 1-3 $\left(\sup_{\substack{0\leq t\leq T}}
\big|Y^n_t-Y^0_t\big|\right)\to0$ in $L^2(\mb{P})$. To sum up, we conclude in particular that
\begin{equation*}\int_0^T(Z_s^n-Z^0_s)^{\tr}\,d\lo
M\ro_s(Z_s^n-Z^0_s)+d\lo N^n-N^0\ro_s
\stackrel{\mb{P}}{\longrightarrow} 0\quad\text{as}\quad n\to+\infty.\qedhere
\end{equation*}
\end{proof}
\begin{rmk}
As previously discussed the sense of convergence given here differs from that in \cite{BH08} where the pointwise convergence of the drivers is assumed, namely
\begin{equation}\label{PointConv}
\mu^A\text{-a.e. for all }y\text{ and }z \text{ we have
}\lim_{n\to+\infty}F^n(\cdot,y,z)= F^0(\cdot,y,z).
\end{equation}
We provide an example in the next section showing that this condition is not sufficient in the present setting so that the statement of \cite{BH08} Proposition 7 needs a small modification.
\end{rmk}

\section{Stability Counterexample}\label{SecStabCount}
Suppose our filtration is the augmentation of the filtration generated by a one dimensional Brownian
motion $W$ so that we may set $A_t=t$ and $B\equiv1$. The measure $\mu^A$ now becomes the product of $\mb{P}$ and the Lebesgue measure on $[0,T]$. In this setting BSDEs take the following form
\begin{equation}
dY_t=Z_t\,dW_t-F(t,Y_t,Z_t)\,dt, \quad Y_T=\xi.\label{BSDEW7}
\end{equation}
A solution now consists of a \emph{pair} $(Y,Z)$ such that $Y$ has
continuous paths, $Z$ is a predictable $W$-integrable process with $\int_0^TZ_t^2\,dt<+\infty$ $\mathbb{P}$-a.s., $\int_0^T|F(t,Y_t,Z_t)|\,dt<+\infty$ $\mathbb{P}$-a.s. and such that
\eqref{BSDEW7} holds, $\mb{P}$-a.s.

Suppose our condition
\[\int_0^T|F^n-F^0|\,(s,Y^0_s,Z^0_s)\,ds\stackrel{\mb{P}}{\longrightarrow}
0\] is replaced by \eqref{PointConv}, i.e. $F^n$ converges to
$F^0$ pointwise $(t,\omega)$-almost everywhere on $[0,T]\times
\Omega$, where the $\mu^A$-null set does not depend on $(y,z)$. One may ask whether this is sufficient for Theorem
\ref{thmstab} to hold, in particular if 
\begin{equation}
\sup_{\substack{0\leq t\leq
T}}|Y_t^n-Y^0_t|\stackrel{\mb{P}}{\longrightarrow}0.\label{BHStab}
\end{equation}
We now present an example to show that this is in fact not the case. The example resembles the standard counterexample to the
dominated convergence theorem and shows that such a stability statement (under the present assumptions) already fails to hold in an
essentially deterministic situation.

Consider $T>1$ together with parameters
$F^0\equiv\a^0\equiv\xi^0\equiv0$. Then all the assumptions in
\cite{BH08} and in the present paper are satisfied and the unique solution to the BSDE
\eqref{BSDEW7} with parameters $(F^0,\xi^0)$ is given by
$(Y^0,Z^0)\equiv 0$, up to appropriate null sets.

Furthermore, for integers $n\geq1$, define the terminal values
$\xi^n\equiv0$ and drivers \[F^n\equiv\a^n\equiv
n\cdot\I{\left[0,\frac{1}{n}\right]\times\Omega}\geq0.\] Observe
that $F^n$ does not depend on $y$ or on $z$ and is constant in
$\omega$, hence deterministic. In particular $|\a^n|_1=\int_0^T\a^n_s\,ds=1$, independently of
$\omega$ and $n$, which shows that again all the assumptions in
\cite{BH08} and in the present paper are satisfied by each pair $(F^n,\xi^n)$.

The unique solution to the BSDE \eqref{BSDEW7} with parameters
$(F^n,\xi^n)$ is given $\mb{P}$-a.s. by $Z^n\equiv0$, more
precisely the zero element in $L^2([0,T]\times\Omega)$, and
\[Y_t^n=(1-nt)\cdot\I{\left[0,\frac{1}{n}\right]\times\Omega}(t,\cdot),\]
which follows from
\[dY_t^n=-n\cdot\I{\left[0,\frac{1}{n}\right]\times\Omega}(t,\cdot)\,dt=-F^n(t)\,dt,\quad
Y_T^n=0=\xi^n,\] together with the $\mb{P}$-a.s. continuity of
$Y^n$. We deduce that
$Y^n$ is nonnegative, nonincreasing and that $Y^n_0=1$,
independent of $n$, $\mb{P}$-a.s.

It follows that, $\mb{P}$-a.s. for all $n\geq1$,
\[\sup_{\substack{0\leq t\leq T}}|Y_t^n-Y^0_t|=Y_0^n=1,\] from which
trivially
\begin{equation}
\lim_{n\to+\infty}\left(\sup_{\substack{0\leq t\leq
T}}|Y_t^n-Y^0_t|\right)=1 \quad\mb{P}\text{-}\mathrm{a.s.}\label{Yconvtone}
\end{equation}
However, by construction, $\lim_{n\to+\infty}F^n=0= F$ on $(0,T]\times \Omega$, hence $\mu^A$-a.e. independently of $y$ and $z$, so that \eqref{PointConv} holds. Since \eqref{BHStab} and \eqref{Yconvtone} cannot hold simultaneously, the condition in \eqref{PointConv} is not sufficient for a stability theorem to hold under the present assumptions. We remark that this phenomenon is not dependent on the non differentiability of the paths of the $Y^n$ as one can choose $F^n$ to be arbitrarily smooth in $t$. Indeed, independent of $\omega$, take a smooth nonnegative function on $[0,T]$ that is identically zero on $(\tfrac{1}{n},T]$ and integrates to one over $[0,T]$. The corresponding $Y^n$ in the BSDE solution are smooth in $t$ and lead to the same contradiction.\\

The problem arising in the proof of \cite{BH08} Proposition 7 can be observed from equations \eqref{C1C2C3} and \eqref{DeltanFs}. More specifically, the authors require $L^2(\mb{P})$-convergence of the random variables $\int_0^T |\Delta^n F(s)|\,ds$ however they only dispose of an estimate on the product space $[0,T]\times\Omega$ of the form \[|\Delta^nF|\leq 2\ol{\beta}|Y^0|+\a^n+\a^0+\gamma\|Z^0\|^2,\] together with uniform integrability assumptions that are on the level of $\Omega$, \emph{with the $t$-component integrated away}. There is no guarantee that the pointwise convergence of $|\Delta^n F|$ on the product space $[0,T]\times\Omega$ will transform to pointwise convergence of the integrals $\int_0^T |\Delta^n F(s)|\,ds$ on $\Omega$, which is necessary to utilize the uniform integrability assumptions. This is the insight behind the present example and motivates the modified condition.\\

We now move on to look at whether the martingale part of our BSDE
solution determines a change of measure.


\section{Change of measure}\label{MeasChange}

In this section we show that under the exponential moments
assumption the martingale part of a solution $(Y,Z,N)$ to the BSDE
\eqref{BSDEF-} defines a measure change. In particular, we do not need to show that $Z\cdot M+N$ is BMO, which is a stronger statement that may indeed not hold, see \cite{FMW11} for some examples and  related discussion. Here we do not
require that 
the driver $F$ be convex in $z$. Our proof is based upon Kazamaki \cite{Ka94} Lemma 1.6 and 1.7 which we state here for martingales on compact time intervals.
\begin{lem}[\cite{Ka94} Lemma 1.6 and 1.7]\label{kazamaki}
If $\widetilde{M}$ is a martingale on $[0,T]$ such that
\begin{equation}\label{eqkaz}
\sup_{\substack{ \tau \mathrm{\,stopping \,time}\\\mathrm{valued \,in\,[0,T]}}}\E\!\left[\exp\!\left(\eta\,\ti{M}_\tau+\left(\frac{1}{2}-\eta\right)\!\lo \widetilde{M}\ro_\tau\right)\right]<+\infty,
\end{equation}
for a real number $\eta\neq1$, then $\mc{E}\!\left(\eta\,\ti{M}\right)$ is a true martingale on $[0,T]$. Moreover, if condition \eqref{eqkaz} holds for \emph{some} $\eta^*>1$ then it holds for \emph{all} $\eta\in(1,\eta^*)$.
\end{lem}
We deduce the following result.
\begin{thm}\label{TrueMart2}
Let Assumption \ref{ass} (iii) hold, $q$ be a real number with
$|q|>\gamma/2$ and $(Y,Z,N)$ a solution to the BSDE \eqref{BSDEF-}
such that, $g$ is bounded by $\gamma/2$, $Y\in\mathfrak{E}$ and
$Z\cdot M+N$ is a martingale. If $\beta>0$ we also require that
$Y^*|\a|_1$ has exponential moments of all orders or that
\eqref{AssLipF} holds with fixed $y_2=0$. Then
$\mc{E}\Big(q\,(Z\cdot M+N)\Big)$ is a true martingale on
$[0,T]$. In particular, when $\gamma<2$,
$\mc{E}(Z\cdot M+N)$ is a true martingale.
\end{thm}
\begin{rmk}
In \cite{MS05} Proposition 7 the authors show that the martingale
part of solutions to the BSDE \eqref{BSDEf} with bounded first
component belongs to the class of BMO martingales so that it yields
a measure change. Our theorem may thus be
seen as a generalization of this result to the case in which $Y$
is not necessarily bounded.
\end{rmk}
\begin{proof}
We apply Lemma \ref{kazamaki} with $\ti{M}:=\tilde{q}(Z\cdot M+N)$ for some fixed $|\tilde{q}|>\gamma/2$. Firstly,
we assume that $\beta>0$ and that $Y^*\an$ has
exponential moments of all orders. Considering \[\log G_\eta(t):=\tilde{q}\eta\,\big[(Z\cdot
M)_t+N_t\big]+\tilde{q}^2\left(\frac12-\eta\right)\lo Z\cdot
M+N\ro_t\] for $\eta>0$
we get from the BSDE \eqref{BSDEF-} and the growth condition in
\eqref{AssGrowF},
\begin{align*}
\log G_\eta(t)&=\tilde{q}\eta\!\left(Y_t-Y_0+ \int_0^t\!F(s,Y_s,Z_s)\,dA_s+\int_0^t\!g_s\,d\lo N\ro_s\right)+\tilde{q}^2\!\left(\frac12-\eta\right)\!\lo Z\cdot
M+N\ro_t \\&\leq2|\tilde{q}|\eta Y^*+|\tilde{q}|\eta|\a|_1+|\tilde{q}|\eta\beta\, Y^*|\a|_1+|\tilde{q}|\eta\left(\frac{\gamma}{2}+\frac{|\tilde{q}|}{\eta}\left(\frac12-\eta\right)\!\right)\lo Z\cdot M+N\ro_t.
\end{align*}
Noting that
\begin{equation*}
\frac{\gamma}{2}+\frac{|\tilde{q}|}{\eta} \left(\frac12-\eta\right)\leq
0\, \Longleftrightarrow\, \eta\geq\frac{|\tilde{q}|}{2|\tilde{q}|-\gamma}=:q_0,
\end{equation*}
we have that $\mb{P}$-a.s. for all $t\in[0,T]$,
\begin{equation*}
G_\eta(t)
\leq\exp\!\left(2|\tilde{q}|\eta Y^*\right)\exp\!\bigg(|\tilde{q}|\eta\an+|\tilde{q}|\eta\beta\,
Y^*\an\bigg),
\end{equation*}
for all $\eta\geq q_0$. By the exponential moments assumption on $Y^*$, $\an$ and $Y^*\an$, we conclude from H\"{o}lder's inequality that
\begin{equation}\label{supG}
\sup_{\substack{\tau \mathrm{\,stopping \,time}\\\mathrm{valued \,in\,[0,T]}}}\E\big[G_\eta(\tau)\big]<+\infty
\end{equation}
for all $\eta\geq q_0>1/2$. It now follows from Lemma \ref{kazamaki} that $\mc{E}(\tilde{q}\eta(Z\cdot M+N))$ is a true martingale for all $\eta \in [q_0,\infty)\backslash\{1\}$. The second part of this lemma ensures that in fact $\mc{E}(\tilde{q}\eta(Z\cdot M+N))$ is a true martingale for all $\eta>1$. Thus, if $|q|>\gamma/2$ we apply this result for some fixed $|\tilde{q}|\in(\gamma/2,|q|)$ and $\eta:=q/\tilde{q}=|q/\tilde{q}|>1$ to conclude that indeed $\mc{E}(q(Z\cdot M+N))$ is a true martingale.\\
Now if $\beta>0$ and \eqref{AssLipF} holds with fixed $y_2=0$, we use \eqref{Ftyzbetas} to derive, similarly to the above,
\begin{equation*}
\log G_\eta(t)\leq2|\tilde{q}|\eta Y^*+|\tilde{q}|\eta|\a|_1+|\tilde{q}|\eta\ol{\beta}\, Y^*A_T+|\tilde{q}|\eta\left(\frac{\gamma}{2}+\frac{|\tilde{q}|}{\eta}\left(\frac12-\eta\right)\!\right)\lo Z\cdot M+N\ro_t
\end{equation*}
so that the claim follows from the boundedness of $A_T$ using exactly the same arguments. The reasoning from above also applies when Assumption \ref{ass} (iii) holds with $\beta=0$, without any further conditions.
\end{proof}


\section{Applications}
In the final section we explore two applications of our results, specifically focusing on utility maximization and partial equilibrium. 
Our contribution is to show that the standard results continue to hold when the usual boundedness assumptions are replaced by appropriate exponential moments conditions, allowing for more generality.

\subsection{Constrained Utility Maximization under Exponential Moments}
In the context of the constrained utility maximization problem with power utility the following BSDE appears, cf. \cite{Mo09} Section 4.2.1,
\begin{gather*}
dY_t=Z_t^{\tr}\,dM_t+dN_t-F(t,Z_t)\,dA_t-\frac{1}{2}\,d\lo N\ro_t,\quad Y_T=0,
\end{gather*}
where the generator is given by
\[F(t,z)=-\frac{p(1-p)}{2}\,\inf_{\substack{\nu\in\mc{C}}}\left\|B_t\!\left(\nu-\frac{z-\lambda_t}{1-p}\right)\right\|^2+\frac{p(1-p)}{2}\left\|B_t\!\left(\frac{z-\lambda_t}{1-p}\right)\right\|^2+\frac12\|B_tz\|^2.\]
In the above $1-p\in(0,+\infty)$ is the investor's relative risk aversion and $\nu$ refers to investment strategies in a stock whose returns are driven by the continuous local martingale $M$ under the market price of risk process $\lambda$ and must be valued in the closed constraint set $\mc{C}$. Writing the infimum in terms of the distance function, which is Lipschitz continuous, one can show that the driver $F$ satisfies Assumption \ref{ass} (ii)-(v), so that there exist constants $c_\lambda$ and $c_z$ such that \[|F(t,z)|\leq c_\lambda\|B_t\lambda_t\|^2+c_z\|B_tz\|^2.\] When we enforce that the \emph{mean-variance trade off} $\int_0^T\lambda_t^\tr\lo M\ro_t\lambda_t$ has all exponential moments, an assumption weaker than that of boundedness given in the cited literature, we are in the current framework and see that the BSDE admits a unique solution in $\mathfrak{E}\times\mathfrak{M}^2\times\mc{M}^2$. The crucial step in \cite{HIM05} and \cite{Mo09} is, given a solution triple $(Y,Z,N)$, to construct the relevant optimizers; this is the process of verification. Building on Theorem \ref{TrueMart2} and not relying on BMO arguments such a verification is performed in \cite{He10}, where we refer the reader for further details and where additional illustration is given via a class of stochastic volatility models. Hence, using the theorems of the present paper, it is possible to show that one can repeat the reasoning of \cite{HIM05} and \cite{Mo09} and that similar results continue to hold for more general classes of market price of risk processes under appropriate trading constraints such as bounded short-selling and borrowing.

A biproduct of the analysis described here is a direct link between solutions to the utility maximization problem and solutions to the BSDE in an exponential moments setting, building on \cite{Nu209}. This allows for a detailed study of the stability of the utility maximization problem, undertaken in \cite{MW10a}, which is important in many applications.
\subsection{Partial Equilibrium and Market Completion under Exponential Moments}
We now briefly describe the partial equilibrium framework of \cite{HPdR10} in which structured securities that are written on nontradable assets are priced via a market clearing condition.

The agents in this economy have preferences which are given by dynamic convex risk measures. The risk they are exposed to is given by two sources. The first is encoded in a financial market in which frictionless trading in a stock $S$ is possible. The second is a non-financial risk factor $R$ that can only be dealt with via a derivative written on this external factor. It is assumed that this derivative completes the market, in fact it is shown that in equilibrium the market is complete.

More specifically, while the market price $\lambda^S$ of financial risk is given exogenously the market price $\lambda^R$ of external risk is determined via an equilibrium condition. This states that when the derivative is priced according to the pricing rule arising from $(\lambda^S,\lambda^R)$ the agents' aggregated demand matches the fixed supply. The demand is in this setting given by the solutions to the agents' individual risk minimization problems and is a function of $\lambda^R$.

To ease the exposition we put ourselves in a representative agent setting where the agent's preferences are of entropic type, i.e. their utility function is exponential. Then the following BSDE for the agent's dynamic risk $Y$ appears
\begin{gather*}
dY_t=Z_t^{\tr}\,dW_t-\frac12\left((\lambda_t^S)^2-2\lambda_t^SZ_t^1-(Z_t^2)^2\right)dt,\quad Y_T=H,
\end{gather*}
where $W$ is a two dimensional Brownian motion representing the two sources of risk, $Z=(Z^1,Z^2)$ is the corresponding control process to $Y$ and $H$ is the agent's endowment. Under suitable exponential moments assumptions the present article provides the existence of a unique solution  $(\hat{Y},\hat{Z})$ to the above BSDE. Once we check that $\hat{Z}^2$ defines a valid pricing rule, i.e. that $\mc{E}\big((\lambda^S,\hat{Z}^2)\cdot W\big)$ is a true martingale, we know that the equilibrium market price $\lambda^R$ of external risk is given by $\lambda^R:=\hat{Z}^2$. To conclude we can generalize \cite{HPdR10}, full details will appear elsewhere.

\bibliography{Exp_Mom_SemiMart_BSDE_A_final}
\bibliographystyle{abbrv}

\end{document}